\documentclass[11pt]{amsart}%
\usepackage{amsmath}
\usepackage{amsfonts}
\usepackage{amssymb}
\usepackage{graphicx}
\usepackage{color}
\usepackage{version}
 \usepackage[all]{xy}%
\usepackage{hyperref}
\usepackage{pinlabel}
\usepackage{mathrsfs}
\usepackage[normalem]{ulem}

\providecommand{\U}[1]{\protect\rule{.1in}{.1in}}

\newtheorem{theorem}{Theorem}[section]
\newtheorem{proposition}[theorem]{Proposition}
\newtheorem{lemma}[theorem]{Lemma}

\newtheorem{corollary}[theorem]{Corollary}

\theoremstyle{definition}
\newtheorem{definition}[theorem]{Definition}

\newtheorem{remark}[theorem]{Remark}

\excludeversion{bib}

\newcommand{\ie}{i.e.\ }
\newcommand{\CSM}{\mathfrak{OB}}
\newcommand{\CM}{\mathfrak{MB}}
\newcommand{\GP}{\mathcal{S}}
\newcommand{\MG}{\mathrm{PSL}(2,\mathbb{Z})}
\newcommand{\PSL}{\mathrm{PSL}}
\newcommand{\PGL}{\mathrm{PGL}}
\newcommand{\GL}{\mathrm{GL}}
\newcommand{\Rep}{\mathrm{Hom}}
\newcommand{\NT}{\Sigma_\lambda}
\newcommand{\PP}{\mathbf{P}}
\newcommand{\IR}{\mathcal{R}^{\circ}}
\newcommand*{\interior}[1]{{#1}^{\circ}}
\newcommand*{\otop}[1]{\check{#1}}
\newcommand*{\trans}[1]{{}^t\!{#1}}

\newcommand{\GG}{\mathscr{G}}
\newcommand{\HH}{\mathscr{H}}
\newcommand{\FF}{\mathscr{F}}

\numberwithin{equation}{section}

\begin{document}

\title[Pappus Theorem, Schwartz Representations and Anosov Representations]{Pappus Theorem, Schwartz Representations \\
and Anosov Representations}

\author{Thierry Barbot}
\address{Universit\'e d'Avignon et des Pays de Vaucluse, Laboratoire de Math\'ematiques, Campus Jean-Henri Fabre, 301, rue Baruch de Spinoza, BP 21239, 84 916 Avignon Cedex 9 France}

\author{Gye-Seon Lee}
\address{Mathematisches Institut, Universit{\"a}t Heidelberg, Im Neuenheimer Feld 205, 69120 Heidelberg, Germany}

\author{Viviane Pardini Val\'erio}
\address{Universidade Federal de S{\~ a}o Jo{\~ a}o del Rei, Departamento de Matem\'atica e Estat\'istica. Pra\c{c}a Frei Orlando, 170, Centro, CEP: 36307-352. S{\~ a}o Jo{\~ a}o del Rei, Minas Gerais, Brazil}

\begin{abstract}
In the paper \textit{Pappus's theorem and the modular group}, R. Schwartz constructed a $2$-dimensional family of faithful representations $\rho_\Theta$ of the modular group $\MG$ into the group $\GG$ of projective symmetries of the projective plane via Pappus Theorem. The image of the unique index $2$ subgroup $\MG_o$ of $\MG$ under each representation $\rho_\Theta$ is in the subgroup $\PGL(3,\mathbb{R})$ of $\GG$ and preserves a topological circle in the flag variety, but $\rho_\Theta$ is \emph{not} Anosov. In her PhD Thesis, V. P. Val\'erio elucidated the Anosov-like feature of Schwartz representations: For every $\rho_\Theta$, there exists a $1$-dimensional family of Anosov representations $\rho^\varepsilon_{\Theta}$ of $\MG_o$ into $\PGL(3,\mathbb{R})$ whose limit is the restriction of $\rho_\Theta$ to $\MG_o$. In this paper, we improve her work: For each $\rho_\Theta$, we build a $2$-dimensional family of Anosov representations of $\MG_o$ into $\PGL(3,\mathbb{R})$ containing $\rho^\varepsilon_{\Theta}$ and a $1$-dimensional subfamily of which can extend to representations of $\MG$ into $\GG$. Schwartz representations are therefore, in a sense, the limits of Anosov representations of $\MG$ into $\GG$.
\end{abstract}



\maketitle


\section{Introduction}
\label{sec:intro}

The initial goal of this work is to understand the similarity between Schwartz representations $\rho_{\Theta}$ of the modular group $\MG$ into the group $\GG$ of projective symmetries, presented in Schwartz \cite[Theorem 2.4]{SCHWARTZ}, and Anosov representations of Gromov-hyperbolic groups, which were studied by Labourie \cite{LABO} and Guichard--Wienhard \cite{GUICHA}.

The starting point is a classical theorem due to Pappus of Alexandria (290 AD - 350 AD) known as Pappus's (hexagon) theorem (see Figure \ref{fig:Pappus}). As said by Schwartz, a slight twist makes this old theorem new again. This twist is to iterate, and thereby Pappus Theorem becomes a dynamical system. An important insight of Schwartz was to describe this dynamic through objects named by him marked boxes. A marked box $[\Theta]$ is simply a collection of points and lines in the projective plane $\PP(V)$ obeying certain rules (see Section \ref{markedbox}). When the Pappus theorem is applied to a marked box, more points and lines are produced, and so on.

\begin{figure}[ht!]
\centering
\includegraphics[scale=0.6]{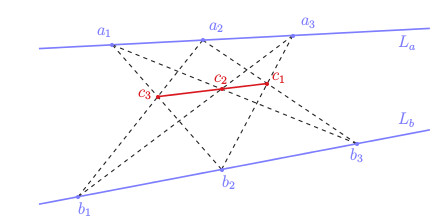}
\caption{\textbf{Pappus Theorem}: If the points $a_1$, $a_2$, $a_3$ are collinear and the points $b_1$, $b_2$, $b_3$ are collinear, then the points $c_1$, $c_2$, $c_3$ are also collinear.}
\label{fig:Pappus}
\end{figure}

The dynamics on the set $\CM$ of marked boxes come from the actions of two special groups $\GG$ and $\mathfrak{G}$. The group $\GG$ of projective symmetries is the group of transformations of the flag variety $\FF$, \ie the group $\GG$ is generated by projective transformations and dualities. The action of $\GG$ on $\CM$ is essentially given by the fact that a marked box is characterized by a collection of flags in $\FF$. The group $\mathfrak{G}$ of elementary transformations of marked boxes is generated by a natural involution $i$, and transformations $\tau_1$, $\tau_2$ induced by Pappus Theorem (see Section \ref{grupoelem}). The group $\mathfrak{G}$ is isomorphic to the modular group $\MG$.

Another Schwartz insight was that for each convex marked box $[\Theta]$, there exists another action of the modular group on the $\mathfrak{G}$-orbit of $[\Theta]$, commuting with the action of $\mathfrak{G}.$ This action can be described in the following way: On the one hand, the isometric action
of the modular group $\MG$ on the hyperbolic plane $\mathbb{H}^2$ preserves the set $\mathcal{L}_o$ of Farey geodesics.  On the other hand, there exists a natural labeling on $\mathcal{L}_o$ by the elements of the $\mathfrak{G}$-orbit of $[\Theta]$. Hence the action of $\MG$ on labels induces an action on the $\mathfrak{G}$-orbit of $[\Theta].$
Moreover, this labeling allows us to better understand how the elements of the $\mathfrak{G}$-orbit of $[\Theta]$ are nested when viewed in the projective plane $\PP(V)$ (or when viewed in the dual projective plane $\PP(V^*)$).

Through these two actions on $\CM$, Schwartz showed that for each convex marked box $[\Theta]$, there exists a faithful representation $\rho_{\Theta} : \MG \to \GG$ such that for every $\gamma$ in $\MG$ and every Farey geodesic $e \in \mathcal{L}_o$, the label of $\gamma(e)$ is the image of the label of $e$ under $\rho_\Theta(\gamma)$ (see Theorem \ref{principalSch}).

As observed in Barbot \cite[Remark 5.13]{BARBOT}, the Schwartz representations $\rho_{\Theta}$, in their dynamical behavior, look like \emph{Anosov representations}, introduced by Labourie \cite{LABO} in order to study the Hitchin component of the space of representations of closed surface groups. Later, Guichard and Wienhard \cite{GUICHA} enlarged this concept to the framework of Gromov-hyperbolic groups, which allows us to define the notion of Anosov representations of $\MG$. Anosov representations currently play an important role in the development of higher Teichm\"uller theory (see e.g.  Bridgeman--Canary--Labourie--Sambarino \cite{BCS}).

In this paper, we show that Schwartz representations are not Anosov, but limits of Anosov representations. More precisely:

\begin{theorem}
\label{thm:main}
Let $\mathcal{R}$ be a region of $\mathbb{R}^2$ given by (\ref{eq:R}) (see Figure \ref{fig:region_R}) with interior $\IR$ and let $\MG_o$ denote the unique subgroup of index $2$ in $\MG$. Then for any convex marked box $[\Theta]$, there exists a two-dimensional family of representations
$$\rho^\lambda_{\Theta} \,:\, \MG_o \to \PGL(3,\mathbb{R}),$$
with $\lambda = (\varepsilon,\delta) \in \mathbb{R}^2$ such that:
\begin{enumerate}
\item If $\lambda = (0,0)$, then $\rho^\lambda_{\Theta}$ coincides with the restriction of the Schwartz representation $\rho_{\Theta}$ to $\MG_o$.
\item If $\lambda \in \mathcal{R}$, then $\rho^\lambda_{\Theta}$ is discrete and faithful.
\item If $\lambda \in \IR$, then $\rho^\lambda_{\Theta}$ is Anosov.
\end{enumerate}
\end{theorem}

Finally, by understanding the extension of the representations $\rho^\lambda_{\Theta} $ to $\MG$, we can prove the following:

\begin{theorem}
\label{thm:main2}
Let $[\Theta]$ be a convex marked box and let $\rho^\lambda_{\Theta} : \MG_o \to \PGL(3,\mathbb{R})$ be the representations as in Theorem \ref{thm:main}. Then there exist a real number $\varepsilon_0 < 0$ and a function $\delta_h : \; ]\varepsilon_0,0] \rightarrow \mathbb{R}$ such that for $\lambda = (\varepsilon, \delta_h(\varepsilon))$,
\begin{enumerate}
\item If $\varepsilon \in \;]\varepsilon_0,0]$, then $\rho_{\Theta}^{\lambda} $ extends naturally to a representation $\bar{\rho}_{\Theta}^{\lambda}$ of $\MG$ into $\GG$.
\item If $\varepsilon = 0$, then $\bar{\rho}^\lambda_{\Theta} = \rho_{\Theta}$.
\item If $\varepsilon \in\;]\varepsilon_0,0[$, then $\rho^\lambda_{\Theta}$ is Anosov.
\end{enumerate}
\end{theorem}

The remainder of this paper is organized as follows.

In Section \ref{sec.Farey}, we recall some basic facts on the Farey triangulation that, as observed by Schwartz, is very useful for the description of the combinatorics of Pappus iterations.
In Section \ref{sec:pappus}, we describe the dynamics on marked boxes generated by Pappus Theorem. In Section \ref{sec:projective}, we introduce the group $\GG$ of projective symmetries and the group $\mathfrak{G}$ of elementary transformations of marked boxes. In Section \ref{sec:Schwartz}, we present Schwartz representations, which involves a labeling on Farey geodesics by the orbit of a marked box under $\mathfrak{G}$. In Section \ref{sec:anosov}, we define Anosov representations. After that, we start the original content of this paper. In Section \ref{sec:new}, we construct our new elementary transformations on marked boxes and our new representations of $\MG_o$ in $\PGL(3,\mathbb{R})$. In Section \ref{sec:norm}, we explain how to define special norms on the projective plane (or its dual plane) for each convex marked box, and using them, in Section \ref{sec:newisanosov}, we prove that our new representations are Anosov, which  establish Theorem \ref{thm:main}. In Section \ref{conclusion}, we understand how to extend new representations to $\MG$ for the proof of Theorem \ref{thm:main2}. Finally, in Section \ref{section:InVariety}, we show that the $\PGL(3,\mathbb{R})$-orbit of our new representations in the algebraic variety $\Rep(\MG_o, \PGL(3,\mathbb{R}))$ has a non-empty interior.

\subsection*{Acknowledgements}
We are thankful for helpful conversations with L\'eo Brunswic, M\'ario Jorge D. Carneiro, S\'ergio Fenley, Nikolai Goussevskii, Fabian Ki{\ss}ler, Jaejeong Lee, Carlos Maquera, Alberto Sarmiento and Anna Wienhard. We also thank the \textit{Funda\c c\~ao de Amparo \`a Pesquisa do Estado de Minas Gerais (FAPEMIG)} and the \textit{Coordena\c c\~ao de Aperfei\c coamento de Pessoal de N\'ivel Superior (CAPES)} for their financial support during the realization of this work. It has been continued during the stay of the first author in the Federal University of Minas Gerais (UFMG) and supported by the program ``Chaire Franco-Br\'esilienne \`a l'UFMG" between the University and the French Embassy in Brazil. The final version of this paper has been elaborated as part of the project MATH AMSUD 2017, Project No. 38888QB - GDAR. Finally, we would like to thank the referees for carefully reading the paper and suggesting several improvements.

G.-S. Lee was supported by the DFG research grant ``Higher Teichm{\"u}ller Theory'' and by the European Research Council under ERC-Consolidator Grant 614733, and he acknowledges support from U.S. National Science Foundation grants DMS 1107452, 1107263, 1107367 ``RNMS: GEometric structures And Representation varieties" (the GEAR Network).

\section{A short review on the Farey triangulation}\label{sec.Farey}
We start by giving a short overview of the Farey graph (see e.g. Katok--Ugarcovici \cite{KU} or Morier-Genoud--Ovsienko--Tabachnikov \cite{OvTab}).

Let $\Delta_{0}$ be the ideal geodesic triangle in the upper half plane $\mathbb{H}^2$ whose vertices are $0, 1, \infty$ in $\partial \mathbb{H}^2$. For any two distinct points $x, y \in \partial \mathbb{H}^2$, we denote by $[x,y]$ the unique oriented geodesic joining $x$ and $y$. The initial point $x$ of an oriented geodesic $e = [x, y]$ is called the \emph{tail} of $e$ and the final point $y$ is called the \emph{head} of $e$. The modular group $\MG$, which is a subgroup of $\PSL(2,\mathbb{R})$, has the group presentation:
\begin{equation}\label{equation:1stpresentation}
\langle\, I,R \mid I^2=1, R^3=1 \,\rangle
\end{equation}
where $I = \left(
\begin{smallmatrix}
0 & 1 \\
-1 & 0 \\
\end{smallmatrix}
\right)
 $ and
 $R = \left(
\begin{smallmatrix}
-1 & 1 \\
-1 & 0 \\
\end{smallmatrix}
\right)
$.
The isometry $R$ of $\mathbb{H}^2$ is the rotation of order 3 whose center is the ``center'' of the triangle $\Delta_{0}$ and that permutes $1$, $0$, $\infty$ in this (clockwise) cyclic order. The isometry $I$ of $\mathbb{H}^2$ is the rotation of order 2 whose center is the orthogonal projection of the ``center'' of  $\Delta_{0}$ on the geodesic $[\infty,0]$ (see Figure \ref{fig:farey_triangulation}).

\begin{remark}\label{rk:index2}
  It will be essential for us to deal with the subgroup $\MG_o$ of $\MG$ generated by $R$ and $IRI$. It consists of the elements of $\MG$ that can be written as a word made up of the letters $I$ and $R$ with an even number of $I$, and it is in fact the {\it unique} index $2$ subgroup of $\MG$ since every homomorphism of $\MG$ into $\mathbb{Z}/2\mathbb{Z}$ must vanish on $R$. Finally, we can also characterize $\MG_o$ as the set of elements of $\MG$ whose trace is an odd integer.
\end{remark}

The Farey graph is a directed graph whose set of vertices is $\mathbb{Q} \cup \{\infty\}$. Taking the
convention $\infty = 1/0$, two vertices $p/q$ and $p'/q'$ (in reduced form) are connected by
an edge if and only if $pq'-p'q  = \pm1$. Two adjacent vertices are connected by exactly two oriented edges $e$ and $\bar{e}$, where $\bar{e}$ is the same as $e$ except the orientation.

It is well-known that if we realize every oriented edge $(p/q, p'/q')$ by the oriented geodesic $[p/q, p'/q']$ in $\mathbb{H}^2$ joining $p/q$ and $p'/q'$ in $\partial\mathbb{H}^2$, then we obtain a triangulation of $\mathbb{H}^2$ (see Figure \ref{fig:farey_triangulation}), called the \emph{Farey triangulation}. Each oriented geodesic in $\mathbb{H}^2$ that realizes an edge of the Farey graph is called a \emph{Farey geodesic} and we denote the set of Farey geodesics by $\mathcal{L}_o$.

\begin{figure}[ht!]
\centering
\includegraphics[scale=0.3]{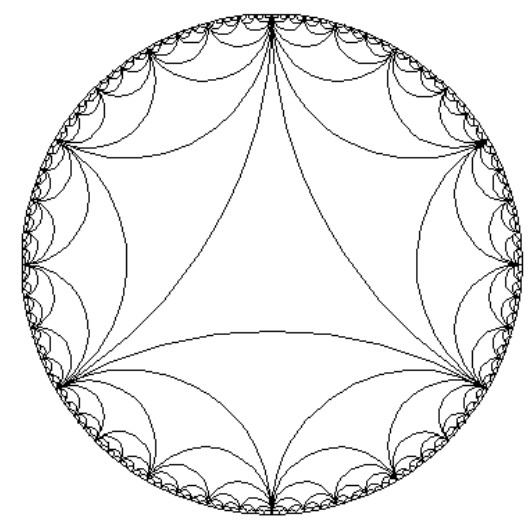}
\footnotesize
\put(-110,26){$0$}
\put(-7,26){$1$}
\put(-60,114){$\infty$}
\put(-56,58){ $R$}
\put(-58,59){$\cdot$}
\put(-60.5,59){\small $\circlearrowright$}
\put(-78,67){$I$}
\put(-69,65){$\cdot$}
\put(-72,65){\small $\circlearrowright$}
\put(-60,46){$\Delta_0$}
\caption{The Farey triangulation in the Poincar{\'e} disk model of $\mathbb{H}^2$}
\label{fig:farey_triangulation}
\end{figure}
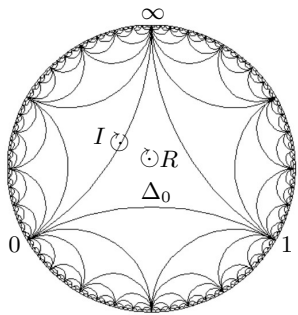

We can regard the Farey triangulation as a tiling: The Farey geodesics without orientation are edges of ideal geodesic triangles, called \emph{Farey triangles}. For example, the triangle $\Delta_{0}$ is a Farey triangle. The modular group $\MG$ acting on $\mathbb{H}^2$ preserves the Farey triangulation, and acts transitively on the set of Farey triangles. Moreover, the stabilizer of $\Delta_0$ in $\MG$ is the subgroup $\langle R \rangle$ of order $3$ generated by $R$.

\begin{remark}\label{rk:index2farey}
 The index $2$ subgroup $\MG_o$ does not act transitively on the set of Farey geodesics, but acts simply transitively on the set of \emph{non-oriented} Farey geodesics. On the other hand, once chosen a Farey geodesic $e_0$ (we will always take $e_0 = [\infty,0]$), then $\MG_o$ acts simply transitively on the orbit of $e_0$ under $\MG_o$, which is called the \emph{$\MG_o$-orientation}.
\end{remark}

It is useful to consider the following other presentation of $\MG$:
\begin{equation}\label{equation:2ndpresentation}
\langle\, I, T_1, T_2  \mid I^2=1, \,\: T_1 I T_2 = I, \,\: T_2 I T_1 = I, \,\:  T_1 I T_1 = T_2, \,\: T_2 I T_2 = T_1 \, \rangle
\end{equation}
where $T_1 :=IR$ and $T_2 := IR^2$.

\begin{remark}\label{rk:eveng}
An element of $\MG$ belongs to the index $2$ subgroup $\MG_o$  if and only if it is a product of an even number of generators $I$, $T_1$, $T_2$.
\end{remark}

Now consider another action of $\MG$ on the set $\mathcal{L}_o$ of Farey geodesics, denoted by $*$, such that for every Farey geodesic $e \in \mathcal{L}_o$, we have (see Figure \ref{fig:click}):

\begin{itemize}
\item $I*e$ is the Farey geodesic $\bar{e}$, which is the same as $e$ except the orientation;
\item $T_1*e$ is the Farey geodesic obtained by rotating $e$ counterclockwise one ``click'' about its tail point;
\item $T_2*e$ is the Farey geodesic obtained by rotating $e$ clockwise one ``click'' about its head point.
\end{itemize}

\begin{figure}[ht!]
\centering
\includegraphics[scale=0.25]{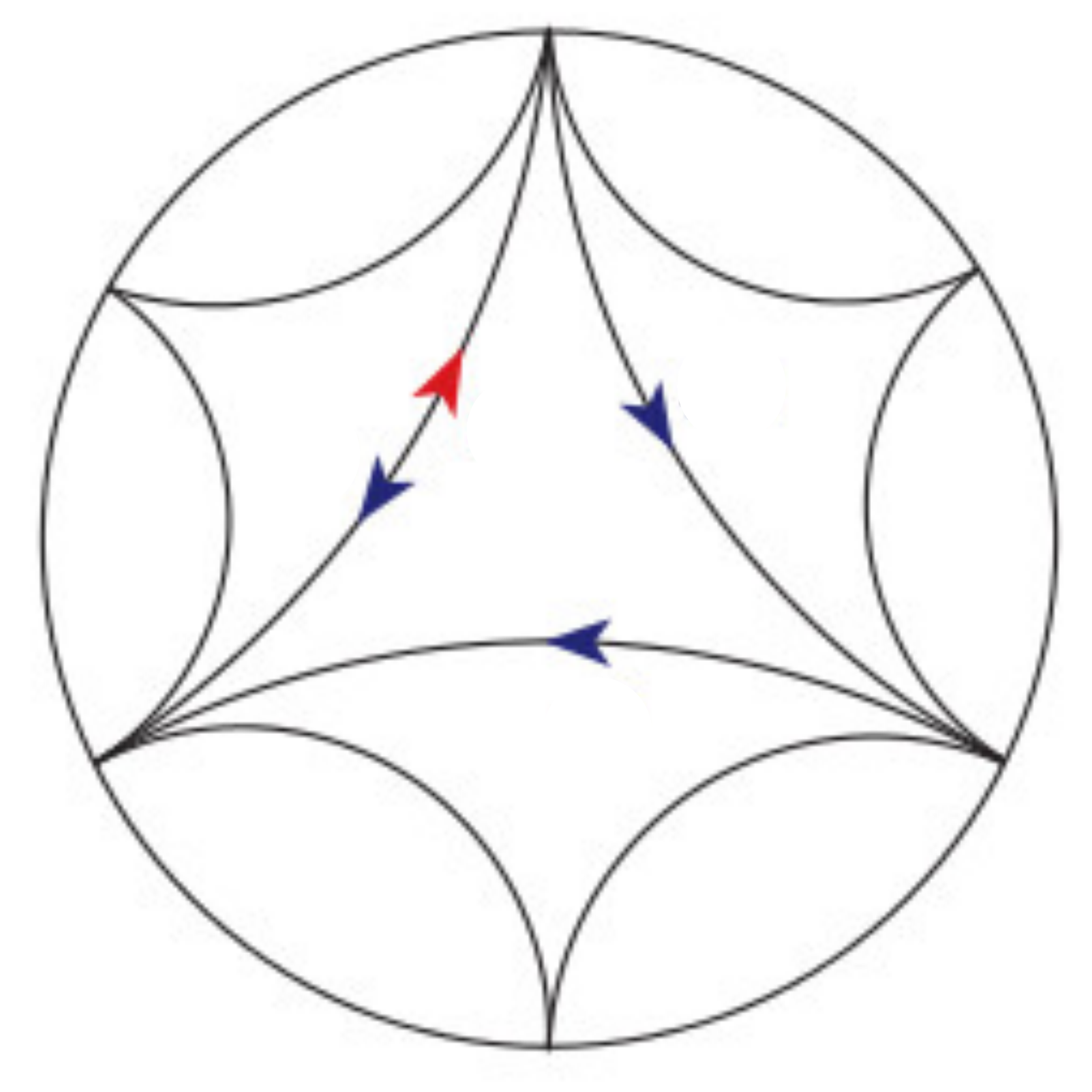}
\footnotesize
\put(-85,66){\textcolor{blue}{$e$}}
\put(-46,74){\textcolor{blue}{$T_1*e$}}
\put(-70,41){\textcolor{blue}{$T_2*e$}}
\put(-93,77){\textcolor{blue}{$I*e$}}
\caption{The $*$-action of $\MG$ on $\mathcal{L}_o$.}
\label{fig:click}
\end{figure}

These two actions of $\MG$ on $\mathcal{L}_o$ are both simply transitive and commute each other.

\begin{remark}
The $*$-action of $\MG$ on $\mathcal{L}_o$ does {\it not} induce an action on the Farey graph since it does {\it not} respect the incidence relation of the graph. For example, even though $T_2 * e$ is incident to $T_2^{-1} * e$, the Farey geodesic $T_1 * (T_2 * e) $ is {\it not} incident to $T_1 * ( T_2^{-1} * e )$.
\end{remark}

\begin{remark}\label{rek:2action}
For the $*$-action, the orbit of a Farey geodesic under $\langle R \rangle$ is the union of the three edges of a Farey triangle because $R = IT_1$ and $R^2=IT_2$. Moreover, for the specific Farey geodesic $e_0 = [\infty,0]$, we have:
$$I(e_0)=I*e_0, \quad R(e_0)=R*e_0, \quad R^2(e_0) = R^2*e_0$$
\end{remark}

\section{A dynamic of Pappus Theorem via marked boxes}
\label{sec:pappus}
As in Schwartz \cite{SCHWARTZ}, we consider the Pappus Theorem as a dynamical system defined on objects called marked boxes. A marked box is essentially a collection of points and lines in the projective plane satisfying the rules that we present below.

\subsection{Pappus Theorem} \label{secPappus}
Let $V$ be a 3-dimensional real vector space and let $\PP(V)$ be the projective space associated to $V$, \ie the space of $1$-dimensional subspaces of $V$. If $a$ and $b$ are two distinct points of $\PP(V)$, then $ab$ denotes the line through $a$ and $b$. In a similar way, if $A$ and $B$ are two distinct lines of $\PP(V)$, then $AB$ denotes the intersection point of $A$ and $B$.

\begin{theorem} [Pappus Theorem]
If the points $a_1$, $a_2$, $a_3$ are collinear and the points $b_1$, $b_2$, $b_3$ are collinear in $\PP(V)$, then the points $c_1= (a_2 b_3) (a_3 b_2)$, $c_2=(a_3 b_1)  (a_1 b_3)$, $c_3=(a_1 b_2) (a_2 b_1)$ are also collinear in $\PP(V)$.
\end{theorem}

We say that the Pappus Theorem is on \emph{generic conditions} if
$a_1$, $a_2$, $a_3$ are distinct points of a line $L_a$, as well as $b_1$, $b_2$, $b_3$ are distinct points of a line $L_b$, and $a_i  \notin L_b$, $b_i \notin L_a$ for all $ i = 1,2,3 $. When the Pappus Theorem is on generic conditions, we have a \emph{Pappus configuration} formed by the points $a_1, a_2, a_3, b_1, b_2, b_3$. An important fact is that the Pappus Theorem on generic conditions can be iterated infinitely many times (see Figure \ref{fig:iteration}), \ie a Pappus configuration is stable, and therefore it gives us a dynamical system (for a proof of the stability of generic conditions under Pappus iteration, see Val\'erio \cite{PV}).

\begin{figure}[ht!]
\centering
\includegraphics[scale=0.65]{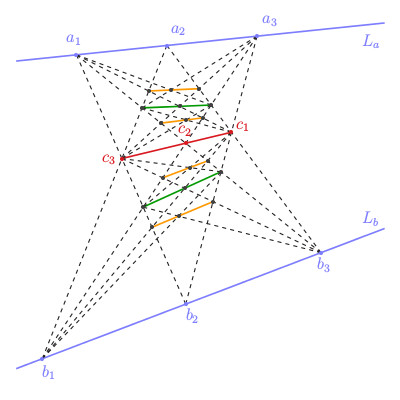}
\caption{An iteration of the Pappus Theorem}
\label{fig:iteration}
\end{figure}

\subsection{Marked boxes} \label{markedbox}
Let $V^*$ be the dual vector space of $V$ and let $\PP(V^*)$ be the projective space associated to $V^*$, \ie the space of lines of $\PP(V)$. An \emph{overmarked box} $\Theta$ of $\PP(V)$ is a pair of distinct 6 tuples having the incidence relations shown in Figure \ref{fig:overmarked_box}:
\begin{eqnarray*}
& \Theta = ((p, q, r, s; t, b),(P, Q, R, S; T, B))\\
& p, q, r, s, t, b \in \PP(V) \, \textrm{ and } \,  P, Q, R, S, T, B \in \PP(V^*)  \\
& P=ts, Q=tr, R=bq, S=bp, T=pq, B=rs \textrm{ and } TB \notin \{p, q, r, s, t, b\}
\end{eqnarray*}

\begin{figure}[ht!]
\centering
\includegraphics[scale=0.75]{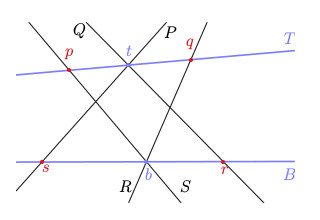}
\caption{An overmarked box in $\PP(V)$}
\label{fig:overmarked_box}
\end{figure}

The overmarked box $\Theta$ is completely determined by the $6$ tuple $(p, q, r, s; t, b)$, but it is wise to keep in mind that
we should treat equally the dual counterpart $(P, Q, R, S; T, B)$.
The \emph{dual} of $\Theta$, denoted by $\Theta^*$, is $((P, Q, R, S; T, B),(p, q, r, s; t, b))$. The \emph{top} (flag) of $\Theta$ is the pair $(t, T)$ and the \emph{bottom} (flag) of $\Theta$ is the pair $(b, B)$.

We denote the set of overmarked boxes by $\CSM$. Let $j: \CSM \to \CSM$ be the involution given by:
\begin{equation}\label{equation:involution}
((p, q, r, s; t, b),(P, Q, R, S; T, B)) \mapsto ((q, p, s, r; t, b),(Q, P, S, R; T, B))
\end{equation}
A \emph{marked box} is an equivalence class of overmarked boxes under this involution $j$. We denote the set of marked boxes by $\CM$. An overmarked box $\Theta = ((p, q, r, s; t, b),(P, Q, R, S; T, B))$ (or a marked box $[\Theta]$) is \emph{convex} if the following two conditions hold:

\begin{itemize}
\item The points $p$ and $q$ separate $t$ and $TB$ on the line $T$.
\item The points $r$ and $s$ separate $b$ and $TB$ on the line $B$.
\end{itemize}

Given a marked box $[\Theta]$, we can define the segments $[pq]$ (resp. $[rs]$) as the closure of the complement in $T$ (resp. $B$) of $\{p, q\}$ (resp. $\{ r, s \}$) containing $t$ (resp. $b$). There are three ways to choose the segments $[qr]$, $[sp]$ simultaneously so that they do not intersect. If the marked box $[\Theta]$ is convex, then one of these three choices leads to a quadrilateral $(p, q, r, s)$ (in this cyclic order)
the boundary of which is not freely homotopic to a line of $\PP(V)$. We then define the \emph{convex interior}, denoted by $\interior{[\Theta]}$, of the convex marked box $[\Theta]$ as the interior of the convex quadrilateral $(p, q, r, s)$ in $\PP(V)$ (see Figure \ref{fig:convex_interior}).

\begin{figure}[ht!]
\centering
\includegraphics[scale=0.75]{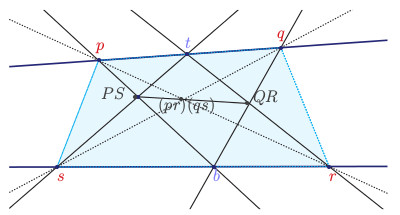}
\caption{A convex interior $\interior{[\Theta]}$ of $[\Theta]$ in $\PP(V)$ is drawn in blue.}
\label{fig:convex_interior}
\end{figure}

Finally, a marked box $[\Theta]$ is convex if and only if the dual $[\Theta^*]$ of $[\Theta]$ is convex, and in this case, we denote by $\interior{[\Theta^*]}$ the convex interior of $[\Theta^*]$ in $\PP(V^*) $. Be careful that $\interior{[\Theta^*]}$ is \emph{not} the convex domain dual to $\interior{[\Theta]}$.

\section{Two groups acting on marked boxes}
\label{sec:projective}

Following Schwartz \cite{SCHWARTZ}, we will explain how the group of projective symmetries acts on marked boxes, and introduce the group of elementary transformations of marked boxes.

\subsection{The group $\GG$ of projective symmetries}
Recall that $V$ is a three-dimensional real vector space and $V^*$ is its dual vector space. We denote by $\langle v^* | v\rangle$ the evaluation of an element $v^*$ of $V^*$ on an element $v$ of $V$. If $W$ is a vector space and $f : V\to W$ is a linear isomorphism between $V$ and $W$, then the dual map of $f$ is the linear isomorphism $f^*: W^*\to V^*$ such that $\langle f^*( w^*) | v\rangle = \langle  w^*| f(v)\rangle$ for all $w^* \in W^*$ and $v \in V$.

We denote the projectivization of $f : V \rightarrow W$ by $\PP(f) : \PP(V) \rightarrow \PP(W)$. A \emph{projective transformation} $T$ of $\PP(V)$ is a transformation of $\PP(V)$ induced by an automorphism $g$ of $V$, \ie $T=\PP(g)$, and the dual map $T^*$ of $T$ is the transformation $\PP(g^*)^{-1}$ of $\PP(V^*)$. A \emph{projective duality} $D$ is a homeomorphism between $\PP(V)$ and $\PP(V^*) $ induced by an isomorphism $h$ between $V$ and $V^*$, \ie $D=\PP(h)$, and the dual map  $D^*$ of $D$ is the homeomorphism $(\PP(h^*) \circ \PP(I))^{-1} : \PP(V^*) \rightarrow \PP(V) $, where $I : V \rightarrow V^{**}$ is the canonical linear isomorphism between $V$ and $V^{**}$. We denote by $[v]$ the $1$-dimensional subspace spanned by a non-zero $v$ of $V$. The \emph{flag variety} $\FF$ is the subset of $\PP(V)\times \PP(V^*) $ formed by all pairs $([v],[v^*])$ satisfying $\langle v^* | v\rangle=0$.

If $T$ is a projective transformation of $\PP(V)$, then there is an automorphism $\mathcal{T}: \FF \to \FF$, also called \emph{projective transformation}, defined by:
\begin{equation}\label{homeo trans}
\mathcal{T}(x, X)=(T(x), T^*(X)) \quad \textrm{for every } (x,X) \in \FF
\end{equation}

Similarly, if $D: \PP(V) \to \PP(V^*)$ is a duality, then there is an automorphism $\mathcal{D}: \FF \to \FF$, also called \emph{duality}, defined by:
\begin{equation}\label{homeo dual}
\mathcal{D}(x, X)=(D^*(X),D(x))\quad \textrm{for every } (x,X) \in \FF
\end{equation}

Let $\HH$ be the set of projective transformations of $\FF$ as in $(\ref{homeo trans})$, and let $\GG$ be the set formed by $\HH$ and dualities of $\FF$ as in $(\ref{homeo dual})$. This set $\GG$ is the \emph{group of projective symmetries} with the obvious composition operation. The subgroup $\HH$ of $\GG$ has index $2$.

\begin{remark}\label{rk:polarity}
 If we equip $V$ with a basis $\mathcal{B}$ and $V^*$ with the dual basis $\mathcal{B}^*$, then the projective space $\PP(V)$ and its dual space $\PP(V^*)$ can be identified with $\PP(\mathbb{R}^3)$. A duality $D: \PP(V) \to \PP(V^*)$ is given by a unique element $A \in \PGL(3,\mathbb{R})$, and the flag transformation $\mathcal{D}$ is expressed by the map
 $$ (x, X) \mapsto (\trans{A}^{-1}X, Ax)$$
 from $\{ (x,X) \in \PP(\mathbb{R}^3) \times \PP(\mathbb{R}^3) \mid x \cdot X =0 \}$ into itself, where $\trans{A}$ denotes the transpose of $A$ and $x \cdot X$ is the dot product of $x$ and $X$. It follows that involutions in $\GG \setminus \HH$ correspond to dualities $D$ for which $A$ is symmetric. They are precisely \emph{polarities}, \ie isomorphims $h$ between $V$ and $V^*$ for which $(u,v) \mapsto \langle h(u) | v\rangle$ is a non-degenerate symmetric bilinear form.
 \end{remark}

\subsection{The action of $\GG$ on marked boxes}

If $T$ is a projective transformation of $\PP(V)$ that induces $\mathcal{T} \in \HH \subset \GG$, then we define a map $\mathcal{T} : \CSM \rightarrow \CSM$ by:
$$
\mathcal{T}(\Theta)=((\hat{p}, \hat{q}, \hat{r}, \hat{s}; \hat{t},\hat{b}),(\hat{P}, \hat{Q}, \hat{R}, \hat{S}; \hat{T}, \hat{B})) \, \textrm{ for every } \Theta \in \CSM
$$
where $\hat{x}=T(x)$ for $x \in \PP(V)$ and $\hat{X}=T^*(X)$ for $X\in \PP(V^*) $.

If $D$ is a duality that induces $\mathcal{D} \in \GG \setminus\HH$, then we define a map $\mathcal{D} : \CSM \rightarrow \CSM$ by:
\begin{equation}\label{equation:reordering}
\mathcal{D}(\Theta)=((P^*, Q^*, {\color{blue}S^*}, {\color{blue}R^*}; T^*,B^*),({\color{blue}q^*}, {\color{blue}p^*}, r^*, s^*; t^*, b^*)) \, \textrm{ for every } \Theta \in \CSM
\end{equation}
where $X^*=D^*(X)$ for $ X \in \PP(V^*)$ and $x^* = D(x)$ for $x \in \PP(V) $.

It is clear that both transformations $\mathcal T$ and $\mathcal D$ commute with the involution $j$ (see (\ref{equation:involution})), and so it induces an action of $\GG$ on $ \CM $, which furthermore preserves the convexity of marked boxes. We will see in Section \ref{grupoelem} that this action commutes with elementary transformations of marked boxes.

\begin{remark}
If we consider the map $\Upsilon: \CSM \to \FF^6$ defined by:
$$\Upsilon((p, q, r, s; t, b),(P, Q, R, S; T, B)) = ( (t,T), (t,P), (t,Q), (b,B), (b,R), (b,S))$$
then it is a bijection onto some subset of $\FF^6$ (which is not useful to describe further). It induces a map
from $\CSM$ into the quotient of $\FF^6$ by the involution permuting the second and the third factor, and the fifth and the sixth factor. Therefore, it gives us a natural action of the group $\GG$ of projective symmetries on $\CM$. In particular, (\ref{equation:reordering}) would be:
\begin{equation}\label{equation:wrong}
\mathcal{D}(\Theta)=((P^*, Q^*, {\color{blue}R^*}, {\color{blue}S^*}; T^*,B^*),({\color{blue}p^*}, {\color{blue}q^*}, r^*, s^*; t^*, b^*)) \, \textrm{ for every } \Theta \in \CSM
\end{equation}
However, as Schwartz observed, (\ref{equation:wrong}) is \emph{not} the one we should consider because with this choice the Schwartz Representation Theorem (Theorem \ref{principalSch}) would fail.
\end{remark}

\begin{remark}\label{rk:notCSM}
Given two dualities $\mathcal{D}_1, \mathcal{D}_2 \in \GG \setminus\HH $, we have:
$$ \mathcal{D}_2 (\mathcal{D} _1 (\Theta)) = j ((\mathcal{D}_2 \mathcal{D}_1) (\Theta)) $$
where $ j: \CSM \to \CSM $ is the involution defining marked boxes (see (\ref{equation:involution})). Hence, the action of $\GG$ on $\CM$ defined by Schwartz does not lift to an action of $ \GG$ on $ \CSM $.
\end{remark}

\subsection{The space of marked boxes modulo $\GG$}
\label{sub:parameter}

Let
$$ \Theta =((p, q, r, s; t, b), (P, Q, R, S; T, B))$$
be an overmarked box.

\begin{definition}
A \emph{$\Theta$-basis} is a basis  of $V$ for which the points $p, q, r, s$ of $\Theta$ have the projective coordinates:
$$ p = [-1 : 1 : 0],\quad q = [1 : 1 : 0],\quad r = [1 : 0 : 1],\quad s = [ -1 : 0 : 1]
$$
If $\zeta_t$ and $\zeta_b$ denote the real numbers, different from $\pm 1$, such that:
$$ t = [\zeta_t : 1 : 0] \quad \textrm{and} \quad b = [\zeta_b : 0 : 1]
$$
then we call $\Theta$ a $(\zeta_t, \zeta_b)$-overmarked box. It is said to be \emph{special} when $(\zeta_t, \zeta_b) = (0, 0)$.
\end{definition}
Observe that for each $\Theta$, there exists a unique $\Theta$-basis up to scaling, and hence that $\zeta_t$, $\zeta_b$ are well-defined. Two overmarked boxes lie in the same $\HH$-orbit  if and only if they have the same coordinates $\zeta_t$ and $\zeta_b$.
In other words, we can identify the space of overmarked boxes modulo $\HH$ with:
$$ \Omega = \{\, (\zeta_t,\zeta_b) \in \mathbb{R}^2 \mid  (\zeta_t^2 - 1)(\zeta_b^2 - 1) \neq 0  \,\} $$
Moreover, the overmarked box $\Theta$ is convex if and only if $\zeta_t$ and $\zeta_b $ are in $]\!-\!1, 1[.$

\begin{remark}
If we equip $V$ with a $\Theta$-basis of $V$ and $V^*$ with its dual basis, then the lines $P, Q, R, S, T, B$ of $\Theta$ have the following projective coordinates:
$$
P = [1 : -\zeta_t : 1], \quad  Q = [-1 : \zeta_t : 1], \quad  R = [-1 : 1 : \zeta_b], \quad  S = [ 1 : 1 : -\zeta_b] $$
$$
T = [0 : 0 : 1] \quad \textrm{and} \quad  B = [0 : 1 : 0]
$$
\end{remark}
\begin{proposition}\label{pro:stabilibox}
Let $\gamma$ be a rotation of $\mathbb{R}^2$ through the angle $\tfrac{\pi}{2}$ about the origin. Then the space of marked boxes modulo $\GG$ is isomorphic to a $2$-dimensional orbifold $\mathcal{O} = \Omega / \langle \gamma \rangle$. In particular, the singular locus of $\mathcal{O}$ consists of a cone point of order $4$, which corresponds to the special marked boxes.
\end{proposition}

\begin{proof}
The involution $j$ maps a $(\zeta_t,\zeta_b)$-overmarked box to a $(-\zeta_t, -\zeta_b)$-overmarked box, and hence the space of marked boxes modulo $\HH$ is isomorphic to:
$$\Omega / \langle - \mathrm{Id} \rangle = \Omega / \langle \gamma^2 \rangle$$
Now, for each $i=1,2$, let
$$
\Theta_i =((p_i, q_i, r_i, s_i; t_i, b_i), (P_i, Q_i, R_i, S_i; T_i, B_i))
$$
be a $(\zeta_{t,i},\zeta_{b,i})$-overmarked box. We claim that $\mathcal{D}(\Theta_1) = \Theta_2$ for some duality $\mathcal{D}$, induced by $D: \PP(V) \to \PP(V^*)$, if and only if:
\begin{equation*}\label{equation:order4symmetry}
\left(
\begin{matrix}
0 & -1 \\
1 & 0 \\
\end{matrix}
\right)
\left(
\begin{matrix}
\zeta_{t, 1} \\
\zeta_{b, 1} \\
\end{matrix}
\right)
=
\left(
\begin{matrix}
\zeta_{t, 2} \\
\zeta_{b, 2} \\
\end{matrix}
\right)
\end{equation*}
Suppose that $\mathcal{D}(\Theta_1) = \Theta_2$. Without loss of generality, we may assume that:
$$ p_1 = p_2, \quad q_1 = q_2, \quad r_1 = r_2, \quad s_1 = s_2 \quad$$
that is, the $\Theta_1$-basis of $V$ is the same as the $\Theta_2$-basis of $V$. Equip $V$ with the $\Theta_i$-basis of $V$ and $V^*$ with its dual basis. The matrix $A$ of the duality $D$ relative to these bases must be:
$$
A = \left(
\begin{matrix}
-1 & \zeta_{b,2} & -\zeta_{t,2} \\
\zeta_{t,2} & -\zeta_{t,2}\zeta_{b,2} & 1 \\
\zeta_{b,2} & -1 & \zeta_{t,2}\zeta_{b,2} \\
\end{matrix}
\right)
\in \GL(3,\mathbb{R})
$$
(up to scaling) because $D : \PP(V) \to \PP(V^*)$ satisfies the following (see (\ref{equation:reordering})):
$$ q_1^* = D(q_1) = P_2, \quad p_1^* = D(p_1) = Q_2, \quad r_1^* = D(r_1) = R_2, \quad s_1^* = D(s_1) = S_2 $$
Moreover, since $t_1^* = D(t_1) = T_2$ and $b_1^* = D(b_1) = B_2$, we have:
$$\zeta_{b,2} = \zeta_{t,1} \quad \textrm{and} \quad \zeta_{t,2} = -\zeta_{b,1}$$
as claimed. Similarly, there exists a duality $\mathcal{D}$ such that $\mathcal{D}(\Theta_1) = j(\Theta_2)$ if and only if:
\begin{equation*}
\left(
\begin{matrix}
0 & 1 \\
-1 & 0 \\
\end{matrix}
\right)
\left(
\begin{matrix}
\zeta_{t, 1} \\
\zeta_{b, 1} \\
\end{matrix}
\right)
=
\left(
\begin{matrix}
\zeta_{t, 2} \\
\zeta_{b, 2} \\
\end{matrix}
\right)
\end{equation*}
Therefore the space of marked boxes modulo $\GG$ is isomorphic to:
$$ \left(\, \Omega  /  \langle \gamma^2 \rangle \,\right) / \langle \gamma \rangle  = \Omega / \langle \gamma \rangle$$
which completes the proof.
\end{proof}

\begin{corollary}\label{cor:box}
In the setting of Proposition \ref{pro:stabilibox}, the space of \emph{convex} marked boxes modulo $\GG$ is isomorphic to a $2$-dimensional orbifold $\,]\!-\!1, 1[\,{}^2 / \langle \gamma \rangle$.
\end{corollary}

\begin{remark}
  The fact that only special marked boxes admit a non-trivial $\GG$-stabilizer is also proved in Barrera--Cano--Navarrete \cite[Lemma 3.1]{Cano}, in a nice, geometric way, involving Desargues' Theorem - in \cite{Cano}, special boxes are called \textit{good boxes}.
\end{remark}

\subsection{The group $\mathfrak{G}$ of elementary transformations of marked boxes} \label{grupoelem}

Let $$\Theta = ((p, q, r, s; t, b), (P, Q, R, S; T, B)) \in  \CSM.$$ The Pappus Theorem gives us two new elements of $\CSM$ that are images of $\Theta$ under two special permutations $\tau_1$ and $\tau_2$ on $\CSM$ (see Figure \ref{fig:tau}). These permutations are defined by:
\begin{eqnarray*}
& \tau_1(\Theta) = \left( (p, q, QR, PS; t, (pr)(qs)),(P, Q, qs, pr; T, (QR)(PS)) \right)\\
& \tau_2(\Theta) = \left( (QR, PS, s, r; (pr)(qs), b),(pr, qs, S, R; (QR)(PS), B) \right)
\end{eqnarray*}

\begin{figure}[ht!]\centering
\includegraphics[scale=0.6]{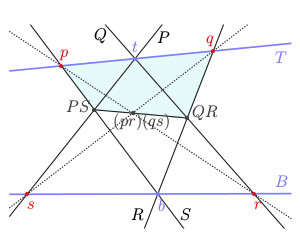}
\quad\quad
\includegraphics[scale=0.6]{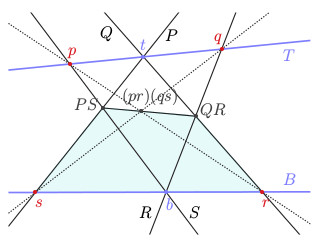}
\caption{Two permutations $\tau_1$ and $\tau_2$; the convex interiors of $\tau_1(\Theta)$ and $\tau_2(\Theta)$ are drawn in blue when $\Theta$ is convex.}
\label{fig:tau}
\end{figure}

There is also a natural involution, denoted by $i$, on $\CSM$ (see Figure \ref{fig:i}) given by:
$$i(\Theta)= ((s, r, p, q; b, t),(R, S, Q, P; B,T))$$

\begin{figure}[ht!]
\centering
\includegraphics[scale=0.65]{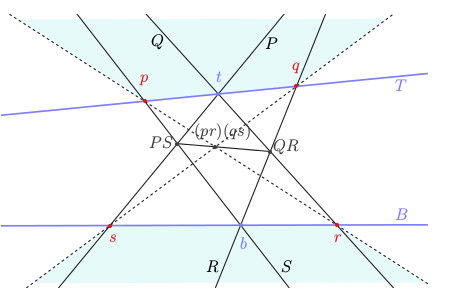}
\caption{The permutation $i$; the convex interior of $i(\Theta)$ is drawn in blue when $\Theta$ is convex.}
\label{fig:i}
\end{figure}

The transformations $ i $, $ \tau_1 $ and $ \tau_2 $ are permutations on $\CSM$ commuting with $j$, hence they also act on $\CM$. We denote by $\GP(\CM)$ the group of permutations on $\CM$.

\begin{remark}\label{rk:nesting}
If $[\Theta]$ is convex, then the new boxes $[i(\Theta)]$, $[\tau_1(\Theta)]$ and $[\tau_2(\Theta)]$ are also convex. The convex interiors of these marked boxes are highlighted in Figures \ref{fig:tau} and \ref{fig:i}. Since:
$$
\interior{[\tau_1(\Theta)]} \subsetneq \interior{[\Theta]}, \quad \interior{[\tau_2(\Theta)]} \subsetneq \interior{[\Theta]} \quad \textrm{and} \quad  \interior{[\tau_1(\Theta)]} \cap \interior{[\tau_2(\Theta)]} = \emptyset
$$
the semigroup of $ \GP(\CM) $ generated by $ \tau_1 $ and $ \tau_2 $ is free. Note that also $\interior{ [i(\Theta)]} \cap \interior{[\Theta]}= \emptyset$.
\end{remark}

\begin{remark} \label{oposto}
In the dual projective plane $\PP(V^*) $, the inclusions are reversed:
$$
  \interior{[\tau_1(\Theta)^*]} \supsetneq \interior{[\Theta^*]}  \quad \textrm{and} \quad \interior{[\tau_2(\Theta)^*]} \supsetneq \interior{[\Theta^*]}
$$
However, we still have $\interior{[i(\Theta)^*]} \cap \interior{[\Theta^*]} = \emptyset$.
\end{remark}

The permutations $i $, $ \tau_1 $ and $ \tau_2 $ on $\CM $ are called \emph{elementary transformations of marked boxes}.
These transformations can be applied iteratively on the elements of $\CM$, so $i $, $ \tau_1 $ and $ \tau_2 $ generate a semigroup $\mathfrak{G}$ of $\GP(\CM$).


\begin{lemma}\label{relacoes}
The following relations hold:
\begin{equation}\label{equation:relations}
i^2=1, \quad \tau_1 i \tau_2 = i, \quad \tau_2 i \tau_1 = i, \quad  \tau_1 i \tau_1 = \tau_2, \quad \tau_2 i \tau_2 = \tau_1
\end{equation}
\end{lemma}

\begin{proof}
See the proof in Schwartz \cite[Lemma 2.3]{SCHWARTZ}.
\end{proof}

Thus, by Lemma \ref{relacoes}, the inverses of $i $, $ \tau_1 $ and $ \tau_2 $ in $\GP(\CM)$ are:
$$
i^{-1}=i,  \quad \tau_1^{-1}=i\tau_2i, \quad \tau_2^{-1}=i\tau_1i
$$
Therefore, the semigroup $\mathfrak{G}$ is in fact a group, and this group $\mathfrak{G}$ is called the \emph{group of elementary transformations of marked boxes}.


\begin{remark} The actions of $\mathfrak{G}$ and $\GG$ on $\CM$ commute each other.
\end{remark}


\begin{lemma}\label{lemma:grouppresentation}
The group $\mathfrak{G}$ has the following presentation:
$$
\langle\, i, \tau_1, \tau_2 \mid  i^2=1,\; \tau_1 i \tau_2 = i,\; \tau_2 i \tau_1 = i,\; \tau_1 i \tau_1 = \tau_2,\; \tau_2 i \tau_2 = \tau_1 \, \rangle
$$
\end{lemma}

\begin{proof}
By Lemma \ref{relacoes}, it only remains to see that (\ref{equation:relations}) is a complete set of relations for the group $\mathfrak{G}$ on the generators $i$, $\tau_1$ and $\tau_2$. Assume that a word $W$ in the symbols $i$, $\tau_1$ and $\tau_2$ is a relator, i.e. it defines the identity element in $\mathfrak{G}$, and that $W$ is \emph{not} derivable from (\ref{equation:relations}). Using the relations in (\ref{equation:relations}), the word $W$ may be reduced to the form $i^awi^b$, where $a, b \in \{0,1\}$ and $w$ is an element of the semigroup generated by $\tau_1$ and $\tau_2$. Since $W$ is a relator, we have that $i^awi^b[\Theta] = [\Theta]$ for every convex marked box $[\Theta]$. Since $\interior{i[\Theta]} \cap \interior{[\Theta]}= \emptyset$, the element $w$ is not trivial. By replacing $[\Theta]$ by $i[\Theta]$, we can further assume that $a=0$. Then we have two cases: either $b=0$ or $b=1$. If $b=0$, then $w[\Theta] = [\Theta]$, which is impossible since $\interior{w[\Theta]} \subsetneq \interior{[\Theta]}$ by Remark \ref{rk:nesting}. If $b=1$, then $wi\interior{[\Theta]}$ is contained in $\interior{i[\Theta]}$, therefore  disjoint from $\interior{[\Theta]}$: contradiction.
\end{proof}

\begin{corollary}\label{corollary:rho}
Let $\varrho_1 = i\tau_1$. Then $\mathfrak{G}$ admits the following group presentation:
$$\mathfrak{G}= \langle\, i, \varrho_1 \mid i^2=1,  \varrho_1^3=1 \,\rangle $$
In particular, it is isomorphic to the modular group $\MG \cong \mathbb{Z}/2\mathbb{Z} \ast \mathbb{Z}/3\mathbb{Z}$.
\end{corollary}

\begin{proof}
It follows from Lemma \ref{lemma:grouppresentation} and (\ref{equation:2ndpresentation}).
\end{proof}

\begin{definition}
We denote by $\Xi : \MG \to  \mathfrak{G}$ the isomorphism given by:
$$\Xi(I)=i, \quad \Xi(T_1)=\tau_1 \quad \textrm{and} \quad \Xi(T_2)=\tau_2$$
(see (\ref{equation:2ndpresentation})). Notice that:
$$\Xi(R) = \varrho_1$$
\end{definition}

\section{Schwartz representations}
\label{sec:Schwartz}

\subsection{Farey geodesics labeled by the $\mathfrak{G}$-orbit of a marked box} \label{geradores}

Given a convex marked box $[\Theta_0] \in \CM$,
we can label each Farey geodesic $e$ by an element $[\Theta](e)$ of the $\mathfrak{G}$-orbit of $[\Theta_0]$ as follows:
first assign the label $[\Theta](e_0)=[\Theta_0]$ for the geodesic $e_0 = [\infty, 0]$,
and then for every geodesic $e = \gamma*e_0$
with $\gamma \in \MG$, define $[\Theta](e) = \Xi(\gamma) [\Theta_0]$.
More generally, for any Farey geodesic $e \in \mathcal{L}_o$ and any $\gamma \in \MG$:
$$[\Theta](\gamma * e) = \Xi(\gamma) [\Theta](e).$$

\begin{remark} \label{aninha} Using this labeling, we can easily see the nesting property of the marked boxes in the $\mathfrak{G}$-orbit of $[\Theta]$ viewed in $\PP(V)$. Assume that the label $[\Theta](e_0)$ of $e_0$ is convex. For each oriented geodesic $e$, we denote by $H_e$ the half space of $\mathbb{H}^2$ on the left of $e$. Let $e$, $e'$ be two Farey geodesics. Then the following property is true: \emph{$H_{e'} \subset H_e$ if and only if the convex interior of $ [\Theta](e')$ is contained in the convex interior of $[\Theta](e)$.} In other words, $H_{e'} \subset H_e$ if and only if $e'$ is obtained from $e$ by applying a sequence of elementary
transformations $\tau_1$ and $\tau_2$. Moreover, $e$ and $e'$ have the same tail point (resp. head point) if and only if the marked boxes $[\Theta](e)$ and $[\Theta](e')$ have the same top (resp. bottom).
\end{remark}

\subsection{Construction of Schwartz representations}

Now we explain how to build Schwartz representations.

\begin{lemma} \label{lemaprinc}
Let $\Theta$ be a convex overmarked box. Then

\begin{enumerate}

\item[1.] there exists a unique projective transformation $\mathcal{A}_{\Theta}^0 \in \HH$ such that:
$$
\Theta \xrightarrow{\mathcal{A}_{\Theta}^0 }  j\varrho_1\Theta \xrightarrow{\mathcal{A}_{\Theta}^0 }  \varrho_1^2\Theta \xrightarrow{\mathcal{A}_{\Theta}^0 } j\Theta
$$

\item[2.] there exists a unique duality $\mathcal{D}_{\Theta}^0  \in \GG \setminus \HH$ such that:
$$
\Theta \xrightarrow{\mathcal{D}_{\Theta}^0 }  j i \Theta \xrightarrow{\mathcal{D}_{\Theta}^0 }  \Theta
$$
\end{enumerate}
Moreover, the duality $\mathcal{D}_{\Theta}^0$ happens to be
a polarity associated to a positive definite quadratic form (see Remark \ref{rk:polarity}).
\end{lemma}

\begin{proof}
The proof is in Schwartz \cite[Theorem $2.4$]{SCHWARTZ} (see also Val{\'e}rio \cite[Lemma $3.1$]{PV} for more details). The uniqueness follows from the fact
that for two overmarked boxes $\Theta_1$ and $\Theta_2$, there exists at most one projective transformation and one polarity respectively mapping $\Theta_1$ to $\Theta_2$. Let us establish the existence. Equip $V$ as usual with the $\Theta$-basis of $V$ and $V^*$ with its dual basis.
Then a straightforward computation shows that the matrix:
\begin{equation}\label{equation:matrixA}
A_\Theta = \left(\begin{array}{ccc}
  \zeta_t\zeta_b-1 & \zeta_t(1-\zeta_t\zeta_b) & \zeta_b-\zeta_t \\
  \zeta_b-\zeta_t & 1-\zeta_t\zeta_b & \zeta_t\zeta_b-1 \\
  0 & 1-\zeta_t^2 & 0
\end{array}\right)
\end{equation}
provides a projective transformation $\mathcal{A}_{\Theta}^0$ as required,
whereas the symmetric matrix:
\begin{equation}\label{equation:matrixD}
D_\Theta = \left(\begin{array}{ccc}
  1 & -\zeta_t & -\zeta_b \\
  -\zeta_t & 1 & \zeta_t\zeta_b \\
  -\zeta_b & \zeta_t\zeta_b & 1
\end{array}\right)
\end{equation}
provides the polarity $\mathcal{D}_{\Theta}^0$, and it is positive definite since $\zeta_t, \zeta_b \in  \,]\!-\!1, 1[\,{}^2$.
\end{proof}

\begin{remark}
At the level of marked boxes, we have:
\begin{align*}
&  [\Theta] \xrightarrow{\mathcal{A}_{\Theta}^0 }  \varrho_1[\Theta] \xrightarrow{\mathcal{A}_{\Theta}^0 }  \varrho_1^2[\Theta] \xrightarrow{\mathcal{A}_{\Theta}^0 } [\Theta]  \\
& [\Theta] \xrightarrow{\mathcal{D}_{\Theta}^0 }  i[\Theta] \xrightarrow{\mathcal{D}_{\Theta}^0 }  [\Theta]
\end{align*}

Since the involution $j$ commutes with projective transformations and polarities, we have:
$$
j\Theta \xrightarrow{\mathcal{A}_{\Theta}^0 }  j\varrho_1j\Theta \xrightarrow{\mathcal{A}_{j\Theta}^0 }  \varrho_1^2j\Theta \xrightarrow{\mathcal{A}_{\Theta}^0 } \Theta
$$
$$
j\Theta \xrightarrow{\mathcal{D}_{\Theta}^0 }  jij\Theta \xrightarrow{\mathcal{D}_{\Theta}^0 }  j\Theta
$$
It follows that $\mathcal{A}_{\Theta}^0$ and $\mathcal{D}_{\Theta}^0$ only depend on the marked box $[\Theta]$.
\end{remark}

\begin{theorem}[\textbf{Schwartz representation Theorem}]\label{principalSch}
Let $[\Theta]$ be a convex marked box, and label the Farey geodesics in $\mathcal{L}_o$ as in Section \ref{geradores} so that $[\Theta](e_0) = [\Theta]$. Then there exists a faithful representation $\rho_\Theta : \MG \to \GG$ such that for every Farey geodesic $e \in \mathcal{L}_o$ and every $\gamma\in \MG$, the following $\rho_\Theta$-equivariant property holds:
$$[\Theta](\gamma e)=\rho_\Theta(\gamma)([\Theta](e))$$
\end{theorem}

\begin{proof} Recall (see (\ref{equation:1stpresentation})) that:
$$ \MG = \langle\, I,R \mid I^2=1, R^3=1 \,\rangle$$
Therefore there exists a representation $\rho_\Theta: \MG \to \GG$ such that:
$$\rho_\Theta(R)=\mathcal{A}_{\Theta}^0 \in \HH \;\; \mbox{ and } \;\; \rho_\Theta(I)=\mathcal{D}_{\Theta}^0 \in \GG \setminus \HH$$
where $\mathcal{A}_{\Theta}^0$ and $\mathcal{D}_{\Theta}^0$ are defined in Lemma \ref{lemaprinc}. Once observed the identities $Re_0 = R*e_0$ and $Ie_0 = I*e_0$ (see Remark \ref{rek:2action}), the $\rho_\Theta$-equivariant property is obviously satisfied for $e = e_0$ and $\gamma = R$ or $I$. Let now $e$ be any other
Farey geodesic. Then:
\begin{eqnarray*}
  [\Theta](Re) &=& [\Theta](R(\gamma*e_0)) \mbox{ for some $\gamma$ in $\MG$}\\
   &=&  [\Theta](\gamma*(Re_0)) \mbox{ (the two actions of $\MG$ on $\mathcal{L}_o$ commute)}\\
   &=&  \Xi(\gamma) [\Theta](Re_0) \mbox{ (by the construction of the labeling on $\mathcal{L}_o$)} \\
   &=&  \Xi(\gamma) \mathcal{A}_\Theta^0([\Theta](e_0)) \mbox{ (the $\rho_\Theta$-equivariant property holds for $\gamma = R, e = e_0$)}\\
   &=& \mathcal{A}_\Theta^0 (\Xi(\gamma) [\Theta](e_0)) \mbox{ (the actions of $\HH$ and $\mathfrak{G}$ on $\CM$ commute)}\\
   &=& \mathcal{A}_\Theta^0 ([\Theta](\gamma*e_0)) \mbox{ (by the construction of the labeling on $\mathcal{L}_o$)}\\
   &=& \rho_\Theta(R)([\Theta](e)) \mbox{ (by the definition of $\rho_\Theta$)}
\end{eqnarray*}
Hence, the $\rho_\Theta$-equivariant property holds for $\gamma = R$ and for every $e \in \mathcal{L}_o$. Similarly, we can check this property for $\gamma = I$, applying the fact that the actions of $\GG$ and $\mathfrak{G}$ on $\CM$ commute for the third-to-last step (whereas \emph{for $\gamma = R$, we only need the fact that $\mathfrak{G}$ commutes with projective transformations}).

Now, the general case follows from the fact that $R$ and $I$ generate $\MG$: If $\gamma$ is an element of $\MG$ for which $[\Theta](\gamma e)=\rho_\Theta(\gamma)([\Theta](e))$ for every $e \in \mathcal{L}_o$, then:
\begin{eqnarray*}
  [\Theta](\gamma Ie) &=& \rho_\Theta(\gamma)([\Theta](Ie)) \\
   &=& \rho_\Theta(\gamma) \rho_\Theta(I)([\Theta](e)) \\
   &=&  \rho_\Theta(\gamma I)([\Theta](e))
\end{eqnarray*}
and similarly $[\Theta](\gamma Re) = \rho_\Theta(\gamma R)([\Theta](e))$, completing the proof by induction on the word length of $\gamma$ in the letters $R$ and $I$.
\end{proof}

We call $\rho_\Theta: \MG \to \GG$ the \emph{Schwartz representation} of $\MG$.

\subsection{The Schwartz map}\label{Schwatzmap}

Recall that in Section \ref{geradores}, for each convex marked box $[\Theta]$, we attach the labels, which are the elements of the orbit of $[\Theta]$ under $\mathfrak{G}$, to the Farey geodesics. As we mentioned in Remark \ref{aninha}, two Farey geodesics have the same tail point in $\partial\mathbb{H}^2$ if and only if the labels of these geodesics are marked boxes with the same top flag. Therefore, it gives us two $\rho_\Theta$-equivariant maps $\varphi: \mathbb{Q} \cup \{ \infty \} \to \PP(V)$ and $\varphi^* : \mathbb{Q} \cup \{ \infty \} \to \PP(V^*)$, and moreover the map $\varphi$ (resp. $\varphi^*$) can be extended to an injective $\rho_\Theta$-equivariant continuous map $\varphi_o : \partial \mathbb{H}^2 \to \PP(V)$ (resp. $\varphi_o^* : \partial \mathbb{H}^2 \to \PP(V^*)$) (see Schwartz \cite[Theorem $3.2$]{SCHWARTZ}). The maps $\varphi_o$ and $\varphi^*_o$ combine to a $\rho_\Theta$-equivariant map, which we call the \emph{Schwartz map}:
$$
\Phi := (\varphi_o,\varphi^{\ast}_o): \partial \mathbb{H}^2 \rightarrow \FF\subset \PP(V)\times \PP(V^*)
$$

\subsection{The case of special marked boxes}
\label{sub:specialcase}

We closely look at the Schwartz representation $\rho_\Theta$ and the Schwartz map $\Phi$ in the case when $[\Theta]$ is a special marked box. In the $\Theta$-basis of $V$, the projective transformation $\mathcal A_\Theta^0$ corresponds to:
$$A = \left(\begin{array}{c|cc}
  1 & 0 & 0 \\
  \hline
  0 & -1 & 1 \\
  0 & -1 & 0
\end{array}\right)$$
whereas the polarity $\mathcal{D}_{\Theta}^0$ is expressed by the identity matrix. Hence, the image of $IRI$ under $\rho_{\Theta}$ corresponds to the inverse of the transpose of $A$:
$$\left(\begin{array}{c|cc}
  1 & 0 & 0 \\
 \hline
  0 & 0 & 1 \\
  0 & -1 & -1
\end{array}\right)$$
Since the elements $R$ and $IRI$ of $\MG_o$ are the equivalence classes of the matrices:
$$
\left(\begin{array}{cc}
   -1 & 1 \\
   -1 & 0
\end{array}\right) \quad \textrm{and} \quad
\left(\begin{array}{cc}
   0 & 1 \\
   -1 & -1
\end{array}\right)
$$
respectively (see (\ref{equation:1stpresentation})), we see that the restriction of $\rho_\Theta$ to $\MG_o$ is a linear action of $\MG_o$ on the affine plane $\{ [x : y : z] \in \PP(V) \mid x\neq0 \}$. It holds true only for $\MG_o$, not for $\MG$: The image of $I$ under $\rho_{\Theta}$ is the polarity associated to an inner product on $V$ for which a $\Theta$-basis of $V$ is orthonormal.

In the case when $[\Theta]$ is special, the map $\varphi_o$ defined in Section \ref{Schwatzmap} is the canonical identification between $\partial\mathbb{H}^2$ and the line $L := \{x = 0 \}$ in $\PP(V)$, and the image of the Schwartz map $\Phi$ is the set of flags $([v], [v^*])$ such that $[v] \in L$ and $[v^*]$ is the line though the points $[v]$ and $[1:0:0]$.

\subsection{Opening the cusps} \label{sub:cocompact}
In the previous subsections, the role of the Farey geodesics is purely combinatorial, except for the definition of the Schwartz map.
We can replace the Farey lamination $\mathcal{L}_o$, which is the set of Farey geodesics, by any other geodesic lamination $\mathcal{L}$ obtained by ``opening the cusps'' in a $3$-fold symmetric way
(see Figure \ref{fig:openingcusp}).
The ideal triangles become hyperideal triangles, which means that these triangles are bounded by three geodesics in $\mathbb{H}^2$, but now these geodesics have no common point in $\partial \mathbb{H}^2$. The lamination $\mathcal{L}$ is still preserved by a discrete subgroup $\Gamma $ of $\mathrm{Isom}(\mathbb{H}^2)$, which is isomorphic to $\MG$ but which is now convex cocompact.

One way to operate this modification is to pick up a hyperideal triangle $\Delta$ containing $\Delta_0$ such that $\Delta$ still admits
the side $e_0 = [\infty, 0]$ but the other two sides are pushed away on the right. The discrete group $\Gamma$ is then generated by $I$
and the unique (clockwise) rotation $R_*$ of order $3$ preserving $\Delta$. Here, we just have to adjust $\Delta$ so that the projection of the ``center'' of the rotation $R_*$ on $e_0 = [\infty, 0]$ is the fixed point of $I$.

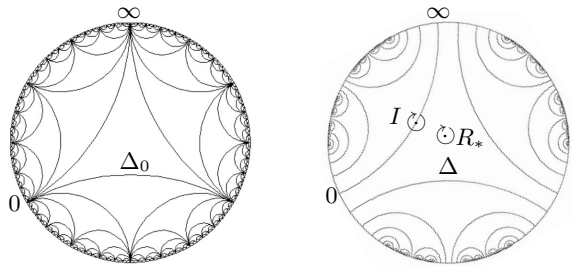
\begin{figure}[ht!]
\centering
\includegraphics[scale=0.25]{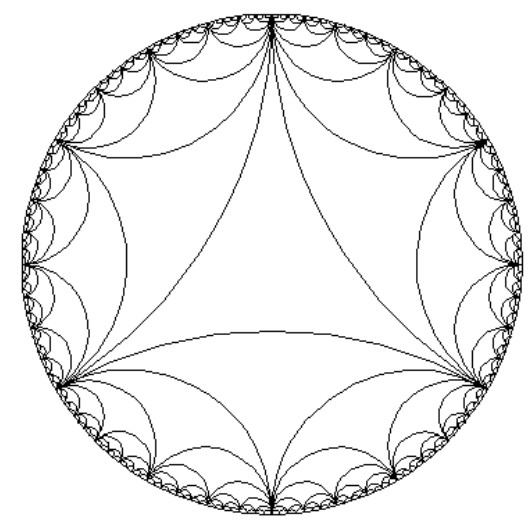}
\quad\quad
\includegraphics[scale=0.20]{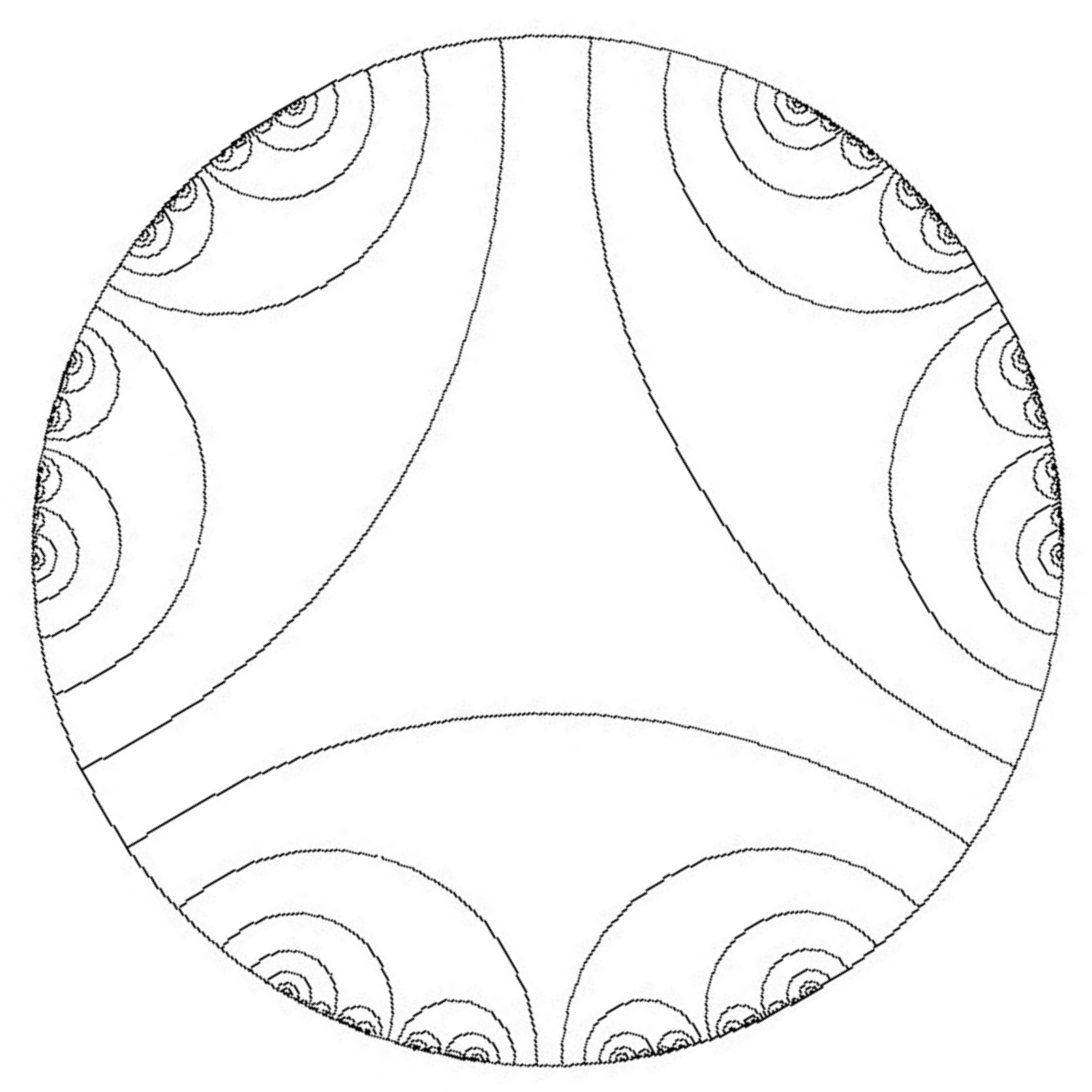}
\footnotesize
\put(-215,22){$0$}
\put(-174,95){$\infty$}
\put(-173,38){$\Delta_0$}
\put(-95,25){$0$}
\put(-57,95){$\infty$}
\put(-49,47){ $R_*$}
\put(-51,48){$\cdot$}
\put(-53.5,48){\small $\circlearrowright$}
\put(-71,55){$I$}
\put(-62,53){$\cdot$}
\put(-64.5,53){\small $\circlearrowright$}
\put(-52.5,36){$\Delta$}
\caption{The Farey lamination $\mathcal{L}_o$ and the new lamination $\mathcal{L}$ obtained by opening the cusps.}
\label{fig:openingcusp}
\end{figure}

All the discussions in the previous subsections remain true if we interpret the notion of ``rotating around the head or tail point'' in the appropriate (and obvious) way.
In particular, in the quotient surface $\Gamma\backslash\mathbb{H}^2$, the leaves of $\mathcal L$ project to wandering geodesics connecting two hyperbolic ends, and for two leaves $e$, $e'$ of $\mathcal{L}$, the labels $[\Theta](e)$ and $[\Theta](e')$ have the same bottom if and only if $e$ and $e'$ have tails in the same connected component of $\partial\mathbb{H}^2 \setminus \Lambda_\Gamma$, where $\Lambda_\Gamma$ is the limit set of $\Gamma$. As a consequence, we still have:

\begin{theorem}[\textbf{Modified Schwartz representation Theorem}]\label{principalSch2}

Let $[\Theta]$ be a convex marked box. Label the oriented leaves of $\mathcal L$, in a way similar to the labeling of $\mathcal{L}_o$ defined in Section \ref{geradores},
 so that $[\Theta](e_0) = [\Theta]$. Then there exists a faithful representation $\rho_\Theta: \Gamma \to \GG$ such that
for every leaf $e \in \mathcal{L}$ and every $\gamma\in \Gamma$ we have:
$$[\Theta](\gamma e)=\rho_\Theta(\gamma)([\Theta](e))$$
\end{theorem}

This modified representation is the one obtained by the original Schwartz representation composed with the obvious isomorphism between $\Gamma$ and $\MG$, and therefore the original and the modified representations are essentially the same. The main difference is that now the $\rho_\Theta$-equivariant map, called the \emph{modified Schwartz map},
$$
\Phi := (\varphi_o,\varphi^{\ast}_o): \Lambda_\Gamma \rightarrow \FF\subset \PP(V)\times \PP(V^*)
$$
obtained by composing the original Schwartz map with the collapsing map $\Lambda_\Gamma \to \partial\mathbb{H}^2$ is not injective: It has the same value on the two extremities of each connected component of $\partial \mathbb{H}^2 \setminus \Lambda_\Gamma$.

\section{Anosov representations}
\label{sec:anosov}

The theory of Anosov representations was introduced by Labourie \cite{LABO} in order to study representations of closed surface groups, and later it was studied by Guichard and Wienhard \cite{GUICHA} for finitely generated Gromov-hyperbolic groups. The definition of Anosov representation involves a pair of equivariant maps from the Gromov boundary of the group into certain compact homogeneous spaces (cf. Barbot \cite{BARBOT}).

The short presentation provided here might appear sophisticated to the uninitiated reader, and the recent alternative definition developed in Bochi--Potrie--Sambarino \cite{Bochi} is more intuitive. However, the definition we select here is more adapted to our proof of Theorem \ref{thm:main}. We try to simplify the definition as much as possible. For example, the ``opening the cusp'' procedure in Section \ref{sub:cocompact} is not really necessary, but has the advantage to realize $\MG$ as a convex cocompact Fuchsian group, so that its Gromov boundary may be identified with the limit set, and to simplify somewhat the definition of Anosov representation.
Moreover, we supply the reader's intuition by stating that the Anosov property of a representation $\rho: \Gamma \to \PGL(3, \mathbb{R})$ means in particular that for every element $\gamma$ of infinite order in $\Gamma$, the image $\rho(\gamma)$ is a loxodromic element, \ie an element of $\PGL(3, \mathbb{R})$ with three real eigenvalues $|\lambda_1| < |\lambda_2| < |\lambda_3|$,
and the ``bigger'' is $\gamma$ in $\Gamma$, the bigger are the ratios $|\lambda_3|/|\lambda_2|$ and $|\lambda_2|/|\lambda_1|$.

\subsection{Definition and properties of Anosov representations}\label{sub:defanosov}

Recall that $V$ is the $3$-dimensional real vector space, and $\PP(V)$ is the real projective plane. Given $x \in \PP(V)$, let $Q_x(V)$ be the space of norms on the tangent space $T_x\PP(V)$ at $x$. Similarly, given $X \in \PP(V^*) $, let  $Q_X(V^*)$ be the space of norms on the tangent space $T_X\PP(V^*) $ at $X$.
Here, a norm is Finsler not necessarily Riemannian. We denote by $Q(V)$ the bundle of base $\PP(V)$ with fiber $Q_x(V)$ over $x\in \PP(V)$, and by $Q(V^*)$ the bundle of base $\PP(V^*) $ with fiber $Q_X(V^*)$ over $X\in \PP(V^*) $.

For each convex cocompact subgroup $\Gamma$ of $\mathrm{PSL}(2,\mathbb{R})$, we denote by $\Lambda_\Gamma$ the limit set of $\Gamma$ and by $\Omega(\phi^t)$ the nonwandering set of the geodesic flow $\phi^t$ on the unit tangent bundle $T^1(\Gamma\backslash\mathbb{H}^2)$ of $\Gamma\backslash\mathbb{H}^2$: It is the projection of the union in $T^1(\mathbb{H}^2)$
of the orbits of the geodesic flow corresponding to geodesics with tail and head in $\Lambda_\Gamma$.

\begin{definition}\label{def:anosov}
Let $\Gamma$ be a convex cocompact subgroup of $\operatorname{PSL}(2,\mathbb{R})$. A homomorphism $\rho: \Gamma \to \PGL(V)$ is a $(\PGL(V), \PP(V))$-\emph{Anosov representation} if there are

\begin{enumerate}

\item[$(i)$] a $\Gamma$-equivariant map
$
\Phi=(\varphi,\varphi^{\ast}): \Lambda_{\Gamma} \to \FF \subset \PP(V) \times \PP(V^*)
$, and
\item[$(ii)$] two maps $\nu_+:\Omega(\phi^t)\subset T^1(\Gamma\backslash\mathbb{H}^2) \to Q(V)$ and $\nu_-:\Omega(\phi^t)\subset T^1(\Gamma\backslash\mathbb{H}^2) \to Q(V^*)$ such that for every $\Gamma$-nonwandering oriented geodesic $c:\mathbb{R}\to \mathbb{H}^2$ joining two points $c_-,c_+\in \Lambda_{\Gamma}$,  the following exponential increasing/decreasing property holds:
\begin{itemize}
\item for every $v\in T_{\varphi(c_+)}\PP(V)$, the size of $v$ for the norm $\nu_+(c(t),c'(t))$ increases exponentially with $t$,
\item for every $v\in T_{\varphi^*(c_-)}\PP(V^*)$, the size of $v$ for the norm $\nu_-(c(t),c'(t))$ decreases exponentially with $t$.
\end{itemize}

\end{enumerate}
\end{definition}

\begin{remark}\label{rk:doubling}
Technically, the norms $\nu_\pm$ in the item $(ii)$ do not need to depend continuously on $(x,v) \in \Omega(\phi^t)$. The continuity, in fact, follows from the exponential increasing/decreasing property. It might be difficult to directly check this property, but there is a simpler criterion: It suffices to prove that there exists a time $T>0$ such that at every time $t$:
  \begin{eqnarray*}
    \nu_+(c(t+T),c'(t+T)) & > & 2 \, \nu_+(c(t),c'(t))\\
    \nu_-(c(t+T),c'(t+T)) & < & \frac{1}{2} \, \nu_-(c(t),c'(t))
  \end{eqnarray*}
For a proof of this folklore, see e.g. Barbot--M{\'e}rigot \cite[Proposition 5.5]{BARBOTMERIGOT}.
\end{remark}

Since the group $\Gamma$ of this definition is a Gromov-hyperbolic group realized as a convex cocompact subgroup of PSL$(2,\mathbb{R})$, its Gromov boundary $\partial \Gamma$ is $\Gamma$-equivariantly homeomorphic to its limit set $\Lambda_{\Gamma}$. The reader can find more information about Gromov-hyperbolic groups in Ghys--de la Harpe \cite{GHYS}, Gromov \cite{GROMOV} and Kapovich--Benakli \cite{KAPO}.

We denote by $\Rep(\Gamma, \PGL(V))$ the space of representations of $\Gamma$ into $\PGL(V)$, and by $\Rep_{A}(\Gamma, \PGL(V))$ the space of Anosov representations in $\Rep(\Gamma, \PGL(V))$.
Here are some basic properties of Anosov representations (see e.g. Barbot \cite{BARBOT}, Guichard--Wienhard \cite{GUICHA} or Labourie \cite{LABO}).

\begin{enumerate}

\item $\Rep_{A}(\Gamma, \PGL(V))$ is an open set in $\Rep(\Gamma, \PGL(V))$.

\item Every Anosov representation is discrete and faithful.

\item The maps $\varphi$ and $\varphi^*$ are injective.

\item For every element of infinite order in $\Gamma$, the image $\rho(\gamma)$ is diagonalizable over $\mathbb{R}$ with eigenvalues that have pairwise distinct moduli.

\item If an Anosov representation is irreducible (\ie it does not preserve a non-trivial linear subspace of $V$), then $\varphi$ (resp. $\varphi^*$) is the unique $\Gamma$-equivariant map from $\partial\Gamma$ into $\PP(V)$ (resp. $\PP(V^*) $).

\end{enumerate}


\subsection{Schwartz representations are not Anosov}

Let $ \rho_\Theta: \Gamma \to \GG $ be the (modified) Schwartz representation associated to a convex marked box $[\Theta]$.
Equip $V$ with the $\Theta$-basis of $V$ and $V^*$ with its dual basis. The projective transformation
 $$ \rho_{\Theta}(T_1^2) =  \mathcal{D}_\Theta^0 \mathcal{A}_{\Theta}^0 \mathcal{D}_\Theta^0 \mathcal{A}_{\Theta}^0$$
corresponds to the matrix:
$$ P_{\Theta} :=   D_\Theta^{-1} \; ^t(A_\Theta)^{-1} D_\Theta \, A_\Theta
$$
where $A_{\Theta}$ and $D_{\Theta}$ are computed in Lemma \ref{lemaprinc}. Then:
\begin{equation*}
Q_\Theta P_{\Theta} Q_{\Theta}^{-1}
= \left(\begin{array}{ccc}
  -1  & 1  & 0 \\
  0   & -1 & 0 \\
  0   & 0  & 1
\end{array}\right)
\quad \textrm{with} \quad
Q_\Theta = \left(\begin{array}{ccc}
  -\zeta_b & 0  & 1 \\
  0        & 2  & 0 \\
  1        & 0  & -\zeta_b
\end{array}\right)
\end{equation*}
As a consequence, the representation $\rho_\Theta$ is not Anosov because it admits non-loxodromic elements and therefore violates the item $(4)$ in Section \ref{sub:defanosov}.

\begin{remark}
There is another fact making clear that $\rho_\Theta$ is not Anosov: if $[\Theta]$ is not special, then
$\rho_\Theta$ is irreducible and the $\Gamma$-equivariant map from $\partial\Gamma$ into $\FF$ must be the map $\Phi$. However, this map is not injective, whereas according to the item (5) in Section \ref{sub:defanosov}, it should be
if $\rho_\Theta$ is Anosov.
\end{remark}

\begin{remark}
One of the referees pointed out that Schwartz representations might be an example of ``relatively'' Anosov representations as currently developed by M. Kapovich (possibly in collaboration with others).
\end{remark}

\section{A new family of representations of $\MG_o$}
\label{sec:new}

In order to show that Schwartz representations are limits of Anosov representations, we build paths (families) of Anosov representations that end in Schwartz representations. With this goal in mind, we first introduce a new group of transformations of marked boxes and consequently we obtain an analog of Theorem \ref{principalSch2} (Schwartz representation Theorem).

\subsection{A new group of transformations of marked boxes}\label{subsec:new}

For each pair $(\varepsilon, \delta)$ of real numbers, we set:
$$\Sigma_{(\varepsilon, \delta)} =
\left( \begin{array}{ccc}
1 & 0 &  0 \\
0 & e^{-\delta} \cosh(\varepsilon)  &  -\sinh(\varepsilon) \\
0 & -\sinh(\varepsilon)                      & e^{\delta} \cosh(\varepsilon)\\
\end{array} \right)
$$
Given an overmarked box $$\Theta = ((p, q, r, s; t, b),(P, Q, R, S; T, B)),$$
choose a $\Theta$-basis of $V$ and define $\sigma_{(\varepsilon, \delta)}(\Theta)$
as the image of $\Theta$ under the projective transformation given in this basis by $\Sigma_{(\varepsilon, \delta)}$. It gives us a new transformation of overmarked boxes $\sigma_{(\varepsilon, \delta)}:\CSM \to \CSM$ (see Figure \ref{fig:new_permutation}).

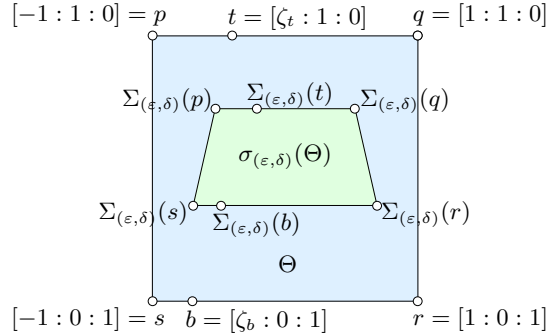
\begin{figure}[ht!]
\centering
\includegraphics[scale=0.3]{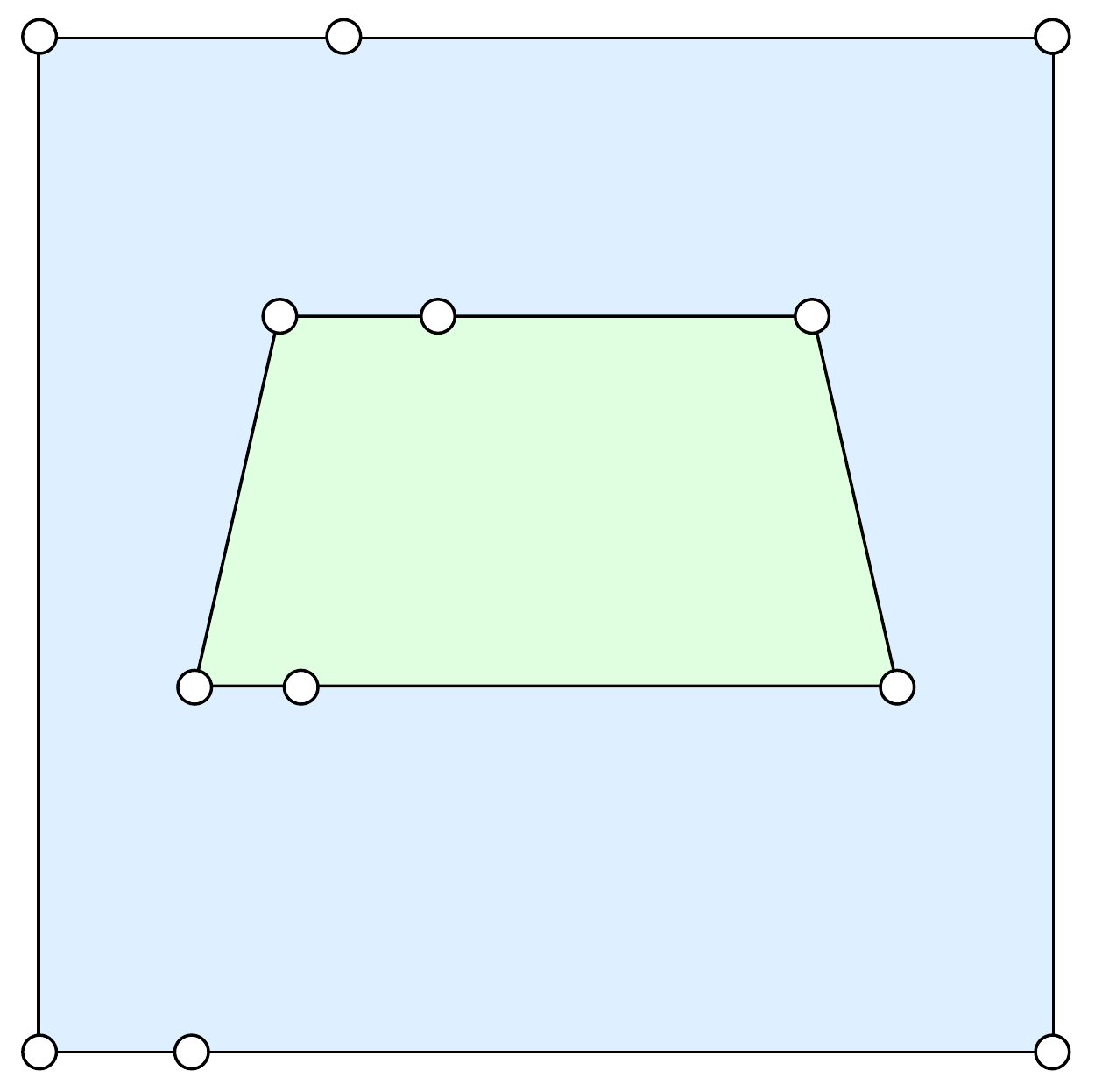}
\footnotesize
\put (-158, 110){$[-1 : 1 : 0] = p$}
\put (-6, 110){$q = [1 : 1 : 0]$}
\put (-6, -5){$r = [1 : 0 : 1]$}
\put (-158, -5){$[-1 : 0 :1] = s$}
\put (-76, 109){$t = [\zeta_t : 1 : 0]$}
\put (-92, -6){$b = [\zeta_b : 0 : 1]$}
\put (-116, 78){$\Sigma_{(\varepsilon, \delta)}(p)$}
\put (-26, 78){$\Sigma_{(\varepsilon, \delta)}(q)$}
\put (-18, 35){$\Sigma_{(\varepsilon, \delta)}(r)$}
\put (-125, 35){$\Sigma_{(\varepsilon, \delta)}(s)$}
\put (-69, 80){$\Sigma_{(\varepsilon, \delta)}(t)$}
\put (-82, 31){$\Sigma_{(\varepsilon, \delta)}(b)$}
\put (-57, 15){$\Theta$}
\put (-72, 57){$\sigma_{(\varepsilon, \delta)}(\Theta)$}
\caption{New permutation $\sigma_{(\varepsilon, \delta)}$ and a convex interior of $\sigma_{(\varepsilon, \delta)}(\Theta)$ in $\PP(V)$ is drawn in green when $\Theta$ is convex.}
\label{fig:new_permutation}
\end{figure}

\begin{lemma}\label{lemma:new_commute}
The transformation $\sigma_{(\varepsilon, \delta)}$ commutes with $j$ and therefore it acts on $\CM$.
\end{lemma}

\begin{proof}
The projective transformation $J$ given by the matrix
$$
\left( \begin{array}{ccc}
-1 & 0 &  0 \\
 0 & 1  & 0 \\
 0 & 0  & 1 \\
\end{array} \right)
$$
for each $\Theta$-basis of $V$
sends the points $p, q, r, s$ onto $q, p, s, r$ (in this order). Thus it induces $j$. It is obvious that $J$ is an involution and $J \Sigma_{(\varepsilon, \delta)} J^{-1} = \Sigma_{(\varepsilon, \delta)}$, and therefore
\begin{align*}
(j \circ \sigma_{(\varepsilon, \delta)} \circ j) (\Theta)
&= (j \circ \sigma_{(\varepsilon, \delta)} \circ j) ((p, q, r, s; t, b),(P, Q, R, S; T, B))\\
&= (j \circ \sigma_{(\varepsilon, \delta)}) ((q, p, s, r; t, b),(Q, P, S, R; T, B))\\
&= j ((\otop{q}, \otop{p}, \otop{s}, \otop{r}; \otop{t}, \otop{b}),(\otop{Q}, \otop{P}, \otop{S}, \otop{R}; \otop{T}, \otop{B}))\\
& = ((\otop{p}, \otop{q}, \otop{r}, \otop{s}; \otop{t}, \otop{b}),(\otop{P}, \otop{Q}, \otop{R}, \otop{S}; \otop{T}, \otop{B}))\\
& = \sigma_{(\varepsilon, \delta)} (\Theta),
\end{align*}
where $\otop{x} = (J \Sigma_{(\varepsilon, \delta)} J^{-1})(x)$ for $x \in \PP(V)$ and $\otop{X} = (J \Sigma_{(\varepsilon, \delta)} J^{-1})^*(X)$ for $X \in \PP(V^*)$.
\end{proof}

\begin{remark}
Every element $T$ of $\HH$ a projective transformation) commutes with $\sigma_{(\varepsilon, \delta)}$ because the image under $T$ of a $\Theta$-basis is a $T(\Theta)$-basis. However, \emph{$\sigma_{(\varepsilon, \delta)}$ does not commute} with elements of $\GG \setminus \HH$ (dualities) acting on $\CM$.
\end{remark}

Recall that the transformation $i$ is the involution on $\CM$ defined in Section \ref{grupoelem}.

\begin{lemma}\label{lemma:great}
The following relations hold:
$$
\sigma_{(-\varepsilon, -\delta)} = \sigma_{(\varepsilon, \delta)}^{-1} \quad \textrm{and} \quad i \sigma_{(\varepsilon, \delta)}= \sigma_{(-\varepsilon, -\delta)} i
$$
\end{lemma}

\begin{proof}
The first relation easily follows from the fact that $\Sigma_{(-\varepsilon, -\delta)} = \Sigma_{(\varepsilon, \delta)}^{-1}$. A proof of the second relation is similar to the proof of Lemma \ref{lemma:new_commute}:
The projective transformation $K$ given by the matrix
$$
\left( \begin{array}{ccc}
-1 & 0 &  0 \\
 0 & 0  & 1 \\
 0 & -1  & 0 \\
\end{array} \right)
$$
for each $\Theta$-basis of $V$
sends the points $p, q, r, s$ onto $s, r, p, q$, respectively. An easy computation shows that $K \Sigma_{(\varepsilon, \delta)} K^{-1} = \Sigma_{(-\varepsilon, -\delta)}$, and therefore
\begin{align*}
(j \circ i \circ \sigma_{(\varepsilon, \delta)} \circ i) (\Theta)
&= (j \circ i \circ \sigma_{(\varepsilon, \delta)} \circ i )((p, q, r, s; t, b),(P, Q, R, S; T, B))\\
& = (j \circ i \circ \sigma_{(\varepsilon, \delta)}) ((s, r, p, q; b, t),(R, S, Q, P; B, T))\\
&= (j \circ i) ((\otop{s}, \otop{r}, \otop{p}, \otop{q}; \otop{b}, \otop{t}),(\otop{R}, \otop{S}, \otop{Q}, \otop{P}; \otop{B}, \otop{T}))\\
& = j((\otop{q}, \otop{p}, \otop{s}, \otop{r}; \otop{t}, \otop{b}),(\otop{Q}, \otop{P}, \otop{S}, \otop{R}; \otop{T}, \otop{B}))\\
& = ((\otop{p}, \otop{q}, \otop{r}, \otop{s}; \otop{t}, \otop{b}),(\otop{P}, \otop{Q}, \otop{R}, \otop{S}; \otop{T}, \otop{B}))\\
& = \sigma_{(-\varepsilon, -\delta)} (\Theta)
\end{align*}
where $\otop{x} = (K \Sigma_{(\varepsilon, \delta)} K^{-1})(x)$ for $x \in \PP(V)$ and $\otop{X} = (K \Sigma_{(\varepsilon, \delta)} K^{-1})^*(X)$ for $X \in \PP(V^*)$.
\end{proof}

Now, define the function $f(\varepsilon, \delta) = e^{-\delta} \cosh(\varepsilon) - \sinh(\varepsilon) -1$ and the region
\begin{equation}\label{eq:R}
\mathcal{R} = \{ (\varepsilon, \delta) \in \mathbb{R}^2 \mid  f(\varepsilon, \delta) \geq 0 \textrm{ and } f(\varepsilon, -\delta) \geq 0 \}
\end{equation}
of $\mathbb{R}^2$ (See Figure \ref{fig:region_R}).

\begin{figure}[ht!]
\centering
\includegraphics[scale=0.5]{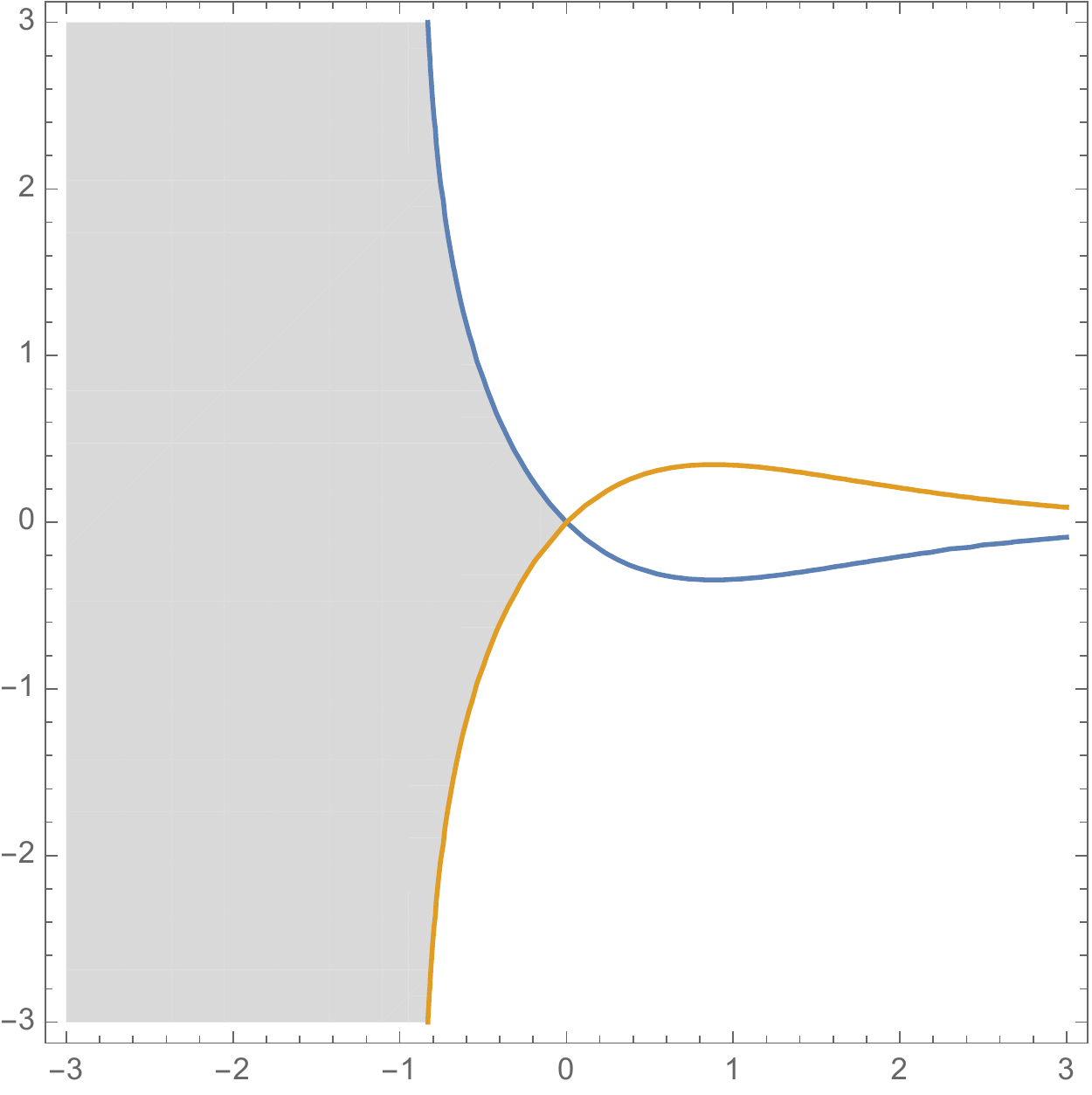}
\footnotesize
\put (-140, 90){$\mathcal{R}$}
\put (-105, 140){$f(\varepsilon,\delta)= 0$}
\put (-105, 47){$f(\varepsilon,-\delta)= 0$}
\caption{The region $\mathcal{R}$ is drawn in grey.}
\label{fig:region_R}
\end{figure}

\begin{proposition}
  For each $(\varepsilon, \delta) \in \mathbb{R}^2$, the convex interior of $\sigma_{\varepsilon, \delta}(\Theta)$ is contained in the convex interior of $\Theta$ if and only if $(\varepsilon, \delta) \in \mathcal{R}$.
\end{proposition}

\begin{proof}
A simple observation is that with respect to the $\Theta$-basis of $V$, the point $[x : y : z] \in \PP(V)$ is in the closure of the convex interior of $\Theta$ if and only if
 \begin{equation}\label{eq:region}
 y + z \neq 0,\quad  \left| \frac{x}{y+z} \right|  \leq  1 \quad \textrm{and} \quad\left| \frac{y-z}{y+z} \right|  \leq  1.
\end{equation}
Therefore, the convex interior of $\sigma_{\varepsilon, \delta}(\Theta)$ is contained in the convex interior of $\Theta$ if and only if the points $\Sigma_{(\varepsilon, \delta)}(p)$ (or $\Sigma_{(\varepsilon, \delta)}(q)$) and $\Sigma_{(\varepsilon, \delta)}(r)$ (or $\Sigma_{(\varepsilon, \delta)}(s)$) satisfy (\ref{eq:region}). The proposition then follows.
\end{proof}

\textbf{\emph{From now on, for the simplicity of the notation, let $\boldsymbol{\lambda = (\varepsilon, \delta)}$.}} For example, $\sigma_{\lambda} = \sigma_{(\varepsilon, \delta)}$. Let us introduce three more new transformations on $\CM$ as follows:
$$
i^\lambda:=\sigma_\lambda i, \quad \tau^\lambda_1 := \sigma_\lambda \tau_1, \quad  \tau^\lambda_2 := \sigma_\lambda \tau_2.
$$

\begin{lemma} \label{relacoes2}
The following relations hold:
$$
({i^\lambda})^2=1, \quad \tau^\lambda_1 i^\lambda\tau^\lambda_2 = i^\lambda, \quad \tau^\lambda_2 i^\lambda\tau^\lambda_1 = i^\lambda, \quad \tau^\lambda_1 i^\lambda\tau^\lambda_1 = \tau^\lambda_2, \quad \tau^\lambda_2 i^\lambda \tau^\lambda_2 = \tau^\lambda_1,  \quad  (i^\lambda\tau^\lambda_1)^3=1.
$$
\end{lemma}

\begin{proof}
The proof follows directly from Lemma \ref{relacoes} and the relation $i \sigma_{\lambda}= \sigma^{-1}_{\lambda} i $.
\end{proof}

Thus, by Lemma \ref{relacoes2}, the inverses of $i^\lambda$, $ \tau_1^\lambda $ and $ \tau_2^\lambda$ are
$$
(i^\lambda)^{-1}=i^\lambda, \quad (\tau_1^\lambda)^{-1}=i^{\lambda}\tau_2^{\lambda}i^{\lambda}, \quad (\tau_2^\lambda)^{-1}=i^\lambda\tau_1^\lambda i^\lambda.
$$

As a result, the semigroup $\mathfrak{G}^\lambda$ of $\GP(\CM)$ generated by $i^\lambda$, $\tau_1^\lambda$ and $\tau_2^\lambda$ is in fact a group. The key point is that if $\lambda \in \mathcal{R}$, then for every convex marked box $[\Theta]$, we still have $\interior{[\tau_1^\lambda(\Theta)]} \subsetneq \interior{[\Theta]}$, $\interior{[\tau_2^\lambda(\Theta)]} \subsetneq \interior{[\Theta]}$ and $\interior{[i^\lambda(\Theta)]} \cap \interior{[\Theta]} = \emptyset$ and furthermore if $\lambda \in \IR$, the interior of $\mathcal{R}$, then we have the same properties but now for the \emph{closures} of the interiors of the marked boxes. The Anosov character of new representations we build is a consequence of this stronger property.

Anyway, by the same arguments as in the case when $\lambda= (0,0)$, we can easily deduce:

\begin{lemma}
The group $\mathfrak{G}^\lambda$ has the following presentation:
$$
\langle\, i^\lambda, \tau_1^\lambda, \tau_2^\lambda \mid  (i^\lambda)^2=1,\; \tau_1^\lambda i^\lambda \tau_2^\lambda = i^\lambda,\; \tau_2^\lambda i^\lambda \tau_1^\lambda = i^\lambda,\; \tau_1^\lambda i^\lambda \tau_1^\lambda = \tau_2^\lambda,\; \tau_2^\lambda i^\lambda \tau_2^\lambda = \tau_1^\lambda \, \rangle
$$
\qed
\end{lemma}

Hence if $\lambda \in \mathcal{R}$, then we have the group presentation:
$$\mathfrak{G}^\lambda = \langle\, i^\lambda, \tau^\lambda_1 \mid  (i^\lambda)^2=1,  (i^\lambda\tau^\lambda_1)^3=1 \,\rangle$$
and thus $\mathfrak{G}^\lambda$ is isomorphic to the modular group. An important corollary of Lemma \ref{lemma:great} is:
\begin{equation}\label{eq:yolo}
  i^\lambda\tau^\lambda_1 = \sigma_\lambda i \sigma_\lambda\tau_1 = i\tau_1
\end{equation}
and so we may rewrite the presentation in the following form:
$$\mathfrak{G}^\lambda = \langle\,  i^\lambda, \varrho_1 \mid (i^\lambda)^2=1,  \varrho_1^3=1 \,\rangle $$
where $\varrho_1 = i\tau_1$ is a Schwartz transformation of marked boxes defined in Corollary \ref{corollary:rho}. In other words, $\mathfrak{G}^\lambda$ is simply obtained from $\mathfrak{G}$ by replacing $i$ by $i^\lambda$, and keeping  $\varrho_1$ the same. As pointed out by a referee, the meaning of (\ref{eq:yolo}) is that the order $3$ projective transformation having the cycle $i(\Theta) \to \tau_1(\Theta) \to \tau_2(\Theta)$ does not change when all three boxes are modified in an equivariant way by $\sigma_\lambda$.

\begin{remark}
If $\lambda \not\in \mathcal{R}$, then the situation is completely different. In this case, it is not clear that $\mathfrak{G}^\lambda$ is isomorphic to $\mathbb{Z}/2\mathbb{Z} \ast \mathbb{Z}/3\mathbb{Z}$. However, it is not important, and in the sequel, when $\lambda \not\in \mathcal{R}$, by $\mathfrak{G}^\lambda$ we mean the group $\mathbb{Z}/2\mathbb{Z} \ast \mathbb{Z}/3\mathbb{Z}$ but acting on the set of marked boxes. Anyway, we are mostly interested in the case when $\lambda \in \IR$ because it corresponds to an Anosov representation.
\end{remark}

\subsection{New representations}

Given a convex marked box $[\Theta]$ and $\lambda = (\varepsilon,\delta) \in \mathbb{R}^2$, let us look at the convex cocompact subgroup $\Gamma$ of $\PSL(2,\mathbb{R})$ and the lamination $\mathcal{L}$ of $\mathbb{H}^2$ introduced in Section \ref{sub:cocompact}, and the new group $\mathfrak{G}^\lambda$ of transformations of $\CM$.

We cannot directly prove an analog of Theorem \ref{principalSch} since it is not true anymore that new transformations of marked boxes commute
with dualities. In order to avoid this inconvenience, we have to restrict the domain of new representations to the subgroup $\Gamma_o$ of $\Gamma$:
$$
\Gamma_o = \langle \, R_*,IR_* I\mid R_*^3=1, (IR_*I)^3=1 \,\rangle.
$$
This subgroup $\Gamma_o$ is isomorphic to $\mathbb{Z}/3\mathbb{Z} \ast \mathbb{Z}/3\mathbb{Z}$, it has index $2$ in $\Gamma$, and it is the image of $\MG_o$ under the isomorphism between $\MG$ and $\Gamma$. It preserves $\mathcal L$ but its action on oriented leaves of $\mathcal L$ is not transitive. However, the action of $\Gamma_o$ on \emph{non-oriented} leaves of $\mathcal L$ is simply transitive. It is also true that the $*$-action of $\Gamma_o$ on the set of non-oriented leaves of $\mathcal L$, which is the restriction of the $*$-action of $\Gamma$, is simply transtive.

In order to define our new representation of $\Gamma_o$ (\emph{not} $\Gamma$), we only need:

\begin{lemma} \label{lemaprinc2}
Let $\Theta$ be a convex overmarked box. Then

\begin{enumerate}

\item[1.] there exists a unique projective transformation $\mathcal{A}_{\Theta}^\lambda \in \HH$ such that:
$$
\Theta \xrightarrow{\mathcal{A}_{\Theta}^\lambda }  ji^\lambda\tau_1^\lambda\Theta \xrightarrow{\mathcal{A}_{\Theta}^\lambda } ( i^\lambda\tau_1^\lambda)^2\Theta \xrightarrow{\mathcal{A}_{\Theta}^\lambda } j\Theta
$$

\item[2.] there exists a unique projective transformation $\mathcal{B}_{\Theta}^\lambda  \in \HH$ such that:
$$
\Theta \xrightarrow{\mathcal{B}_{\Theta}^\lambda }  j\tau_1^\lambda i^\lambda\Theta \xrightarrow{\mathcal{B}_{\Theta}^\lambda } ( \tau_1^\lambda i^\lambda)^2\Theta \xrightarrow{\mathcal{B}_{\Theta}^\lambda } j\Theta
$$

\end{enumerate}

\end{lemma}

\begin{proof}
  The first item is exactly the first item of Lemma \ref{lemaprinc} since $i^\lambda\tau_1^\lambda = i\tau_1 = \varrho_1$, hence $\mathcal{A}_{\Theta}^\lambda$ is precisely $\mathcal{A}_{\Theta}^0$.

  The second item is a corollary of the first item: apply the first item to $i^\lambda\Theta$, and use the fact that $i^\lambda$ commutes with $\mathcal{A}_{\Theta}^\lambda$. However, we give an alternative proof, which is useful for the later discussion: If we recall that $\NT$ is the projective transformation of $\PP(V)$ defined in Section \ref{subsec:new} and $\mathcal{B}_{\Theta}^0$ is the image of $IR_*I$ under $\rho_\Theta$ in Theorem \ref{principalSch2}, then the projective transformation $\mathcal{B}_{\Theta}^\lambda$ is actually $\NT^{-1}\mathcal{B}_{\Theta}^0\NT$.

  Let $\varrho'_1 = \tau_1i$ and look at the following diagram, which arises from the fact that $\NT^{-1}$ commutes with every elementary transformation of marked boxes:

   $$ \xymatrix{\relax
     \Theta \ar[r]^{\mathcal{B}_{\Theta}^0} \ar[d]_{\NT^{-1}}  & j\varrho'_1\Theta \ar[d] \ar[r]^{\mathcal{B}_{\Theta}^0} \ar[d]_{\NT^{-1}} & (\varrho'_1)^2\Theta \ar[r]^{\mathcal{B}_{\Theta}^0} \ar[d]_{\NT^{-1}} & j\Theta \ar[d]_{\NT^{-1}}\\
    \sigma_\lambda^{-1}\Theta  & j\varrho'_1\sigma_\lambda^{-1}\Theta & (\varrho'_1)^2\sigma_\lambda^{-1}\Theta & j\sigma_\lambda^{-1}\Theta
    } $$
   Therefore:
   $$ \xymatrix@!C{\relax
      \sigma_\lambda^{-1}\Theta  \ar[r]^-{\NT^{-1}\mathcal{B}_{\Theta}^0\NT}  & j\varrho'_1\sigma_\lambda^{-1}\Theta  \ar[r]^-{\NT^{-1}\mathcal{B}_{\Theta}^0\NT} & (\varrho'_1)^2\sigma_\lambda^{-1}\Theta \ar[r]^-{\NT^{-1}\mathcal{B}_{\Theta}^0\NT} & j\sigma_\lambda^{-1}\Theta
          } $$
  Since $\sigma_\lambda$ and the projective transformation $\NT^{-1}\mathcal{B}_{\Theta}^0\NT$ commute each other:
  $$ \xymatrix@!C{\relax
     \Theta  \ar[r]^-{\NT^{-1}\mathcal{B}_{\Theta}^0\NT}  & j\sigma_\lambda\varrho'_1\sigma_\lambda^{-1}\Theta  \ar[r]^-{\NT^{-1}\mathcal{B}_{\Theta}^0\NT} & \sigma_\lambda(\varrho'_1)^2\sigma_\lambda^{-1}\Theta \ar[r]^-{\NT^{-1}\mathcal{B}_{\Theta}^0\NT} & j\Theta
          } $$
  Our claim then follows because:
  \begin{eqnarray*}
    \sigma_\lambda\varrho'_1\sigma_\lambda^{-1} &=& \sigma_\lambda\tau_1i\sigma_\lambda^{-1} \\
     &=& \tau_1^\lambda(\sigma_\lambda i) \\
     &=& \tau_1^\lambda i^\lambda
  \end{eqnarray*}
\end{proof}

The next Theorem is similar to Theorem \ref{principalSch} (better to say, Theorem \ref{principalSch2}), but now the leaves of $\mathcal L$ must be understood as \emph{non-oriented} geodesics.

\begin{theorem}\label{principalSch3}
Let $[\Theta]$ be a convex marked box and let $\lambda \in \mathbb{R}^2$. Then there exists a representation
$\rho^\lambda_\Theta: \Gamma_o \to \HH \subset \GG$
such that for every (non-oriented) leaf $e$ of $\mathcal L$ and every $\gamma\in \Gamma$ we have:
$$[\Theta](\gamma e)=\rho^\lambda_\Theta(\gamma)([\Theta](e))$$
Moreover, if $\lambda \in \mathcal{R}$, then $\rho^\lambda_\Theta$ is faithful.
\end{theorem}

\begin{proof} Define $\rho^\lambda_\Theta: \Gamma_o \to \HH \subset \GG$ by requiring:
$$
\rho_\Theta^\lambda(R_*) = \mathcal{A}_{\Theta}^\lambda \in \HH \quad \textrm{and} \quad \rho_\Theta^\lambda(IR_*I)=\mathcal{B}_{\Theta}^\lambda \in \HH
$$
where $\mathcal{A}_{\Theta}^\lambda$ and $\mathcal{B}_{\Theta}^\lambda$ are the projective transformations defined in Lemma \ref{lemaprinc2}.
Here we can apply the arguments in the proof of Theorem \ref{principalSch} since no dualities are involved - in that proof, we emphasized that the commutativity between the actions of $\mathfrak{G}$ and dualities was used only for defining the image of the involution $I$.
\end{proof}

\section{A special norm associated to marked boxes} \label{norma}
\label{sec:norm}
In this section, we will show that given a convex marked box $[\Theta]$, we can define a special norm associated to $[\Theta]$. For this purpose, we use the Hilbert metric on properly convex domains. The reader can find more information about the Hilbert metric in Marquis \cite{LUDOVIC} or Orenstein \cite{OP}.

Let $D$ be a \emph{properly convex} domain in $\PP(V)$, \ie there exists an affine chart $\mathbb{A}$ of $\PP(V)$ such that the closure $\overline{D}$ of $D$ is contained in $\mathbb{A}$ and $D$ is convex in $\mathbb{A}$ in the usual sense.  For distinct points $x,y \in D$, let $p$ and $q$ be the intersection points of the line $xy$ with the boundary $\partial D$ in such a way that $a$ and $y$ separate $x$ and $b$ on the line $xy$ (see Figure \ref{fig:hilbert_metric}).
The \emph{Hilbert metric} $d^h_D: D \times D \to [0,+\infty)$ is defined by:
$$
d^h_D(x,y)=\frac{1}{2} \log \left( [a:x:y:b] \right) \; \textrm{ for every } x \neq y \in D \quad \textrm{and} \quad d^h_D(x,x) = 0
$$
where $[a:x:y:b]$ is the cross-ratio of the four points $a, x, y, b \in \PP(V)$. More precisely, $[a:x:y:b] := \frac{|a-y|}{|a-x|} \cdot \frac{|b-x|}{|b-y|}$ for any Euclidean norm $| \cdot |$ on any affine chart $\mathbb{A}$ containing $\overline{D}$.

The Hilbert metric can be also defined by a Finsler norm on the tangent space $T_x D$ at each point $x \in D$: Let $x \in D$, $v \in T_xD$ and let $p^+$ (resp. $p^-$) be the intersection point of $\partial D$ with the half-line determined by $x$ and $v$ (resp. $-v$) (see Figure \ref{fig:hilbert_metric}).

\begin{figure}[ht!]
\centering
\includegraphics[scale=0.6]{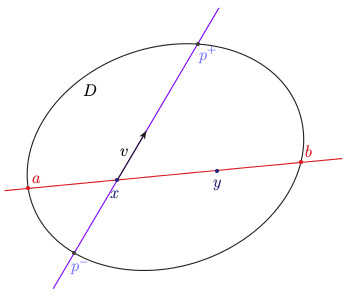}
\caption{The Hilbert metric}
\label{fig:hilbert_metric}
\end{figure}

The \emph{Hilbert norm} of $v$, denoted by $\|v\|^h_D$, is the Finsler norm defined by:
$$
\|v\|^h_D := \frac{|v|}{2} \left( \frac{1}{|x - p^-|}+ \frac{1}{|x - p^+|}\right)  = \frac{d}{dt}\bigg|_{t=0} d^h_D(x,x + tv)
$$
The following lemma demonstrates the expansion property of the Hilbert metric by inclusion:

\begin{lemma} \label{expansion}
Let $D_1$ and $D_2$ be properly convex domains in $\PP(V)$. Assume that $\overline{D_2}\subset D_1$. Then there exists a constant $C > 1$ such that

\begin{enumerate}

\item $d^h_{D_2}(x,y)$ $\geq$ $C\,d^h_{D_1}(x,y)$  for every $x,y\in D_2$,

\item $\|v\|^h_{D_2}$ $\geq$ $C\, \|v\|^h_{D_1}$ for every  $x\in D_2$ and for every $v\in T_x{D_2} = T_x{D_1}$.
\end{enumerate}
\end{lemma}

\begin{proof}
See the proof in Orenstein \cite[Teorema 7]{OP}.
\end{proof}

\begin{definition}\label{def:distortion}
  The \emph{distortion} from $D_1$ to $D_2$, denoted by $C(D_2, D_1)$, is the upper bound of the set of $C$'s for which (1) and (2) in Lemma \ref{expansion} hold.
\end{definition}

The following lemma is obvious since projective transformations preserve the cross-ratio.

\begin{lemma}\label{le:cpreserved}
  Let $D_1$ and $D_2$ be two properly convex domains in $\PP(V)$ such that $\overline{D_2}\subset D_1$, and let $g$ be a projective
  transformation of $\PP(V)$. Then $C(gD_2, gD_1) = C(D_2, D_1)$.
  \qed
\end{lemma}

Moreover:

\begin{lemma}\label{le:cdecrease}
  Let $D_1$, $D_2$, $D_3$ be  properly convex domains in $\PP(V)$ such that $\overline{D_2}\subset D_1$ and $\overline{D_3}\subset D_2$.
  Then $C(D_3, D_1) \geq C(D_3, D_2)\,C(D_2,D_1)$.
  \qed
\end{lemma}

\begin{remark} For each convex marked box $[\Theta]$, the convex interior of $[\Theta]$ (resp. $\interior{[\Theta^*]}$) is a properly convex domain in $\PP(V)$ (resp. $\PP(V^*)$). Hence we can define the Hilbert metric $($norm$)$ on $\interior{[\Theta]}$ (resp. $\interior{[\Theta^*]}$).
\end{remark}

\section{A family of Anosov representations}
\label{sec:newisanosov}
In this section, we give the proof of Theorem \ref{thm:main}. Recall that we can identify $\HH$ with $\PGL(V)$ and in Theorem \ref{principalSch3} we define the representations $\rho^\lambda_\Theta: \Gamma_o \to \HH$. Since the groups $\MG_o$ and $\Gamma_o$ are isomorphic, we just have to show that the representations
$\rho^\lambda_\Theta$ are Anosov when $\lambda \in \IR$.

From now on, assume that $\lambda \in \IR$.
We only need to verify that there exist
\begin{enumerate}

\item  a $\Gamma_o$-equivariant map $\Phi^\lambda=(\varphi_\lambda,\varphi^{\ast}_\lambda): \Lambda_{\Gamma_o} \to \FF\subset \PP(V) \times \PP(V^*) $, and

\item  two maps  $\nu_+:\Omega(\phi^t)\subset T^1(\Gamma_o \backslash\mathbb{H}^2) \to Q(V)$ and $\nu_-:\Omega(\phi^t)\subset T^1(\Gamma_o \backslash \mathbb{H}^2) \to Q(V^*)$ that ``carry'' the Anosov property of expansion and contraction.

\end{enumerate}

\subsection{Combinatorics of the geodesic flow with respect to $\mathcal{L}$}\label{sub:combinatorics}

Let $\alpha \in \Lambda_{\Gamma_o}$ and let $c$ be the $\Gamma_o$-nonwandering oriented geodesic whose head is $\alpha$. Since $c$ is nonwandering, it meets infinitely many leaves of $\mathcal L$. We orient each of these leaves so that $c$ crosses each of them from the right to the left, and denote them by $\ell_m$ with $m \in \mathbb{Z}$ (see Figure \ref{fig:sequence_L}).
\begin{figure}[ht!]
\centering
\includegraphics[scale=0.7]{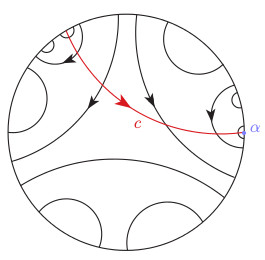}
\scriptsize
\put(-103,93){$\ell_{-1}$}
\put(-87,73){$\ell_0$}
\put(-54,76){$\ell_1$}
\caption{A sequence $(\ell_m)$ of oriented leaves of $\mathcal{L}$ crossed by $c$}
\label{fig:sequence_L}
\end{figure}

Recall the objects $e_0$ and $\Delta$ in Section \ref{sub:cocompact}.
We can assume without loss of generality that the leaf $\ell_0$ is the image of $e_0$ under some element $\gamma_0$ of $\Gamma_o$. Now, we forget all the other leaves with odd index.

\begin{lemma}\label{lemma:W}
For every integer $n$, the oriented leaf $\ell_{2n}$ is in the $\Gamma_o$-orbit of $\ell_0$. Furthermore, if $\gamma_n$ is the unique element of $\Gamma_o$ for which $\ell_{2n} = \gamma_n e_0$, then we have
$\gamma_{n+1} = \gamma_nw$, where $w$ is one of the elements of the following subset of $\Gamma_o$:
$$W = \{ R_*IR_*I, \;\;  R_*IR_*^2I, \;\; R_*^2IR_*^2I, \;\;  R_*^2IR_*I \} $$
\end{lemma}

\begin{proof}
The image of $c$ under $\gamma_n^{-1}$ crosses $e_0$ from the right to the left, hence enters in $\Delta$. Then it exits $\Delta$ from one of the two other sides $R_*e_0$  or $R_*^2e_0$ (see Figure \ref{fig:four_leaves}). Observe that these crossings are both from the left to the right, hence
$\ell_{2n+1}$ is the image under $\gamma_n$ of $R_*e_0$ or $R_*^{-1}e_0$ \emph{with the reversed orientation}, and therefore $\ell_{2n+1}$ is not
in the orbit of $e_0$ under $\Gamma_o$.

In the first case, the case when the geodesic crosses $R_*e_0$, it enters in the triangle $R_*IR_*^2 \Delta$, and then exit, from the right to the left, through either $R_*IR_*Ie_0$ or $R_*IR_*^2Ie_0$. Thus, we obtain that $\gamma_{n+1} = \gamma_nR_*IR_*I$ or $\gamma_{n+1} = \gamma_nR_*IR_*^2I$.

In the second case, we just have to replace $R_*$ by $R_*^{2}$, and we then have $\gamma_{n+1} = \gamma_nR_*^2IR_*^2I$ or $\gamma_{n+1} = \gamma_nR_*^2IR_*I$. The result follows.
\end{proof}

\begin{figure}[ht!]
\centering
\includegraphics[scale=0.4]{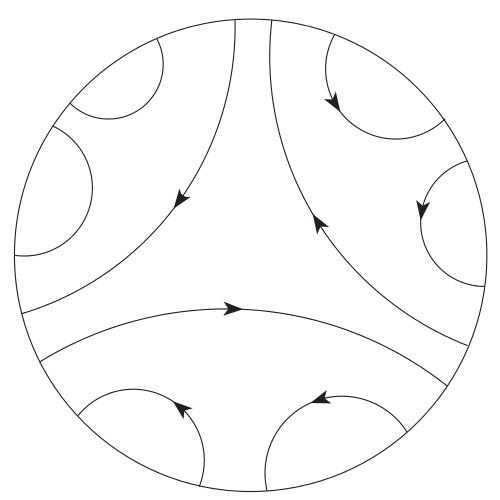}
\footnotesize
\put(-90,82){$e_0$}
\put(-145,51){$0$}
\put(-81,141){$\infty$}
\put(-61,16){$R^2_*IR_*I e_0$}
\put(-128,17){$R^2_*IR^2_*I e_0$}
\put(-21,77){$R_*IR^2_*I e_0$}
\put(-46,114){$R_*IR_*Ie_0$}
\put(-65,85){$R_*e_0$}
\put(-128,17){$R^2_*IR^2_*I e_0$}
\put(-86,61){$R^2_*e_0$}
\put(-77,77){$\Delta$}
\caption{The four leaves on the left of $e_0$ at the $2$nd step}
\label{fig:four_leaves}
\end{figure}
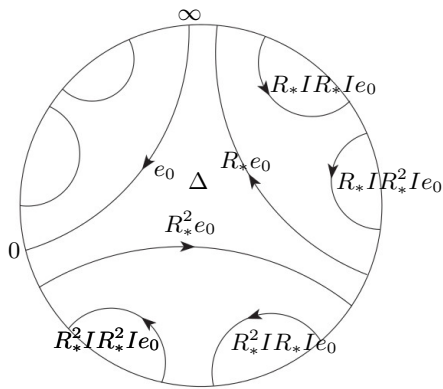

\begin{definition}\label{def:C}
The minimal distortion after crossing two leaves of $\mathcal{L}$ is (recall Definition \ref{def:distortion}):
\begin{equation*}\label{uniform_bound}
C := \mathrm{min}\left\{\, C( \interior{ [\Theta](we_0) }, \interior{[\Theta](e_0)}) \mid w \in W \,\right\}
\end{equation*}
\end{definition}

\subsection{The equivariant map of new representations}\label{sub:varphi}

In this section we prove:

\begin{proposition}\label{pro:constructphi}
  There exists  a $\Gamma_o$-equivariant continuous map:
  $$\Phi^\lambda=(\varphi_\lambda,\varphi^{\ast}_\lambda): \Lambda_{\Gamma_o} \to \FF\subset \PP(V) \times \PP(V^*)$$
\end{proposition}

We will construct $\varphi_\lambda:\Lambda_{\Gamma_o}\to \PP(V)$ and $\varphi_\lambda^\ast:\Lambda_{\Gamma_o}\to \PP(V^*) $ separately. First label the oriented leaves of $\mathcal L$ by elements of the orbit of $[\Theta]$ under $\mathfrak{G}^\lambda$ as in Section \ref{geradores}. Again, let $\alpha \in \Lambda_{\Gamma_o}$ and let $c$ be a $\Gamma_o$-nonwandering oriented geodesic whose head is $\alpha$. Consider as in Section \ref{sub:combinatorics} the sequence $(\ell_m)_{m \in \mathbb Z}$ of oriented leaves of $\mathcal{L}$ succesively crossed by $c$.

According to Remark \ref{aninha}, the labels $[\Theta](\ell_m)$ of these oriented leaves $\ell_m$ of $\mathcal{L}$ give us a sequence of convex marked boxes $[\Theta_m] := [\Theta](\ell_m)$ satisfying the following nesting property in $\PP(V)$:
\begin{equation} \label{seq}
\dotsc \supset  \interior{[\Theta_{-1}]} \supset \interior{[\Theta_{0}]} \supset \interior{[\Theta_{1}]}  \supset \dotsc \supset \interior{[\Theta_{m}]} \supset \dotsc
\end{equation}

\begin{lemma}\label{pro:intersection}
 The intersection $\bigcap_{m \in \mathbb{Z}} \interior{[\Theta_m]} = \bigcap_{n \in \mathbb{Z}} \interior{[\Theta_{2n}]}$ is reduced to a single point in $\PP(V)$. Moreover, this intersection
 is the same for all geodesics $c$ with head $\alpha$.
\end{lemma}

\begin{proof}
It follows from Lemmas \ref{le:cpreserved} and \ref{lemma:W} that for every $n \in \mathbb{Z}$, we have:
$$C(\interior{[\Theta_{2n+2}]}, \interior{[\Theta_{2n}]}) \geq C$$
where $C$ is the constant defined in Definition \ref{def:C}. If we look at all the closures of the convex domains $\interior{[\Theta_{2n}]}$, then it is a decreasing sequence
of compact sets as $n$ goes to infinity, and hence their intersection is not empty.

Assume that the intersection contains two different elements $a$, $b$. Let $x$, $y$ be two distinct elements in
the segment $]a, b[$. For every integer $n$, let $d^h_n(x,y)$ be the Hilbert metric between $x$ and $y$ with respect to the domain $\interior{[\Theta_{2n}]}$. According to Lemma \ref{le:cdecrease}, we
have:
$$d^h_n(x,y) \geq C^n d^h_0(x,y)$$
On the other hand, for every $n$, we have:
$$ d^h_n(x,y) \leq \frac{1}{2} \log ( [a:x:y:b])$$
which is a contradiction.
Therefore, the intersection $\bigcap_{m \in \mathbb{Z}} \interior{[\Theta_m]} = \bigcap_{n \in \mathbb{Z}} \interior{[\Theta_{2n}]}$ is reduced to a single point in $\PP(V)$.

Moreover, since any other nonwandering oriented geodesic $c'$ with head $\alpha$ ultimately intersects the same leaves of $\mathcal L$, the intersection of the labels of leaves of $\mathcal{L}$ crossed by $c'$ is the same as for $c$. The lemma then follows.
\end{proof}

The previous Lemma provides:

\begin{definition}
For any $\Gamma_o$-nonwandering oriented geodesic $c$, we denote by $\psi_\lambda(c)$ the unique intersection point of the convex interiors of the labels of leaves of $\mathcal L$ crossed by $c$. Define a map $\varphi_\lambda:\Lambda_{\Gamma_o}\to \PP(V)$ by assigning to $\alpha \in \Lambda_{\Gamma_o}$ the point $\psi_\lambda(c)$ where $c$ is any geodesic with head $\alpha$.
\end{definition}

\begin{lemma}
  The map $\varphi_\lambda:\Lambda_{\Gamma_o}\to \PP(V)$ is continuous.
\end{lemma}

\begin{proof}
Let $\mathcal{U}$ be any open neighborhood of $\varphi_\lambda(\alpha)$ in $\PP(V)$. Then there exists a marked box $[\Theta_{2n}]$ such that $\varphi_\lambda(\alpha) \in \interior{[\Theta_{2n}]} \subset \mathcal{U}$ since $\bigcap_{n \in \mathbb{Z}} \interior{[\Theta_{2n}]}$ is a singleton. Hence, if $\beta \in \Lambda_{\Gamma_o}$ is sufficiently close to $\alpha$, then every geodesic with head $\beta$ will intersect $\ell_{2n}$, and thus $\varphi_\lambda(\beta)$ is contained in the interior of $[\Theta_{2n}]$.
\end{proof}

In a similar way, we define the map $\varphi_\lambda^\ast:\Lambda_{\Gamma_o}\to \PP(V^*) $. By Remark \ref{oposto}, the inclusions of the sequence (\ref{seq}) along the oriented nonwandering geodesic $c$ are reversed when viewed in $\PP(V^*)$:
\begin{equation*} \label{seq2}
\dotsc  \subset \interior{[\Theta_{-1}^*]} \subset \interior{[\Theta_{0}^*]} \subset \interior{[\Theta_{1}^*]}  \subset  \dotsc \subset \interior{[\Theta_{m}^*]} \subset \dotsc
\end{equation*}
We can show that this nested sequence of convex domains is again uniform with respect to the Hilbert metrics; in particular, the
intersection $\bigcap_{m \in \mathbb{Z}} \interior{[\Theta_{m}^*]}$ is reduced to a single point in $\PP(V^*)$, and two nonwandering geodesics $c$ and $c'$ sharing the same \emph{tail} $\alpha$ leads to the same point. Thus, it provide:

\begin{definition}
For any $\Gamma_o$-nonwandering oriented geodesic $c$, we denote by $\psi^*_\lambda(c)$ the unique intersection point of the convex interiors of the dual marked boxes of the labels of leaves of $\mathcal L$ crossed by $c$. Define a map $\varphi^*_\lambda:\Lambda_{\Gamma_o}\to \PP(V^*)$ by assigning to $\alpha \in \Lambda_{\Gamma_o}$ the point $\psi^*_\lambda(c)$ where $c$ is any geodesic with tail $\alpha$.
\end{definition}

The maps $\varphi_\lambda$ and $\varphi_\lambda^*$ are obviously $\Gamma_o$-equivariant, but it is not clear from our construction that they combine to a map in the flag variety, \ie that $\varphi_\lambda(\alpha)$ is a point in the line $\varphi_\lambda^*(\alpha)$ of $\PP(V)$. However, a simple trick, which we describe now, makes it obvious.

We work in the setting of the proof of Proposition \ref{pro:intersection}: Let $\bar{\ell}_n$ be the leaf $\ell_n$ with the reversed orientation, \ie $\bar{\ell}_n = I*\ell_n$ (see Figure \ref{fig:sequence_iL}). Then the dual labels
$$([\Theta](\bar{\ell}_n))^* = ([\Theta](I * \ell_n))^* =(\Xi(I)[\Theta](\ell_n))^* = [i(\Theta_{n})^*]$$
form a nested sequence:
\begin{equation*}
\dotsc \supset \interior{[i(\Theta_{-1})^*]} \supset \interior{[i(\Theta_{0})^*]} \supset \interior{[i(\Theta_{1})^*]} \supset  \dotsc \supset \interior{[i(\Theta_{m})^*]} \supset  \dotsc
\end{equation*}
The common intersection point $\bigcap_{m \in \mathbb{Z}} \interior{[i(\Theta_{m})^*]}$ is clearly $\psi_\lambda^*(\bar{c})$, where $\bar{c}$
is the geodesic $c$ with the reversed orientation. In particular, if $\alpha$ is the head of $c$, then this intersection point is $\varphi^*_\lambda(\alpha)$.

Now the key point is that the top point $t^*_m$ of each $[i(\Theta_{m})^*]$ is the bottom line of $[\Theta_m]$. The bottom points $b_m$ of $[\Theta_m]$ converge to $\psi_\lambda(c)$ whereas $t^*_m$ converge to $\psi^*_\lambda(\bar{c})$. Since every $t^*_m$ contains $b_m$, the line
 $\varphi_\lambda^*(\alpha)$ of $\PP(V)$ also contains $\varphi_\lambda(\alpha)$.
 Hence, the maps $\varphi_\lambda$ and $\varphi_\lambda^\ast$ combine to a $\Gamma_o$-equivariant map:
$$\Phi^\lambda=(\varphi_\lambda,\varphi^{\ast}_\lambda): \Lambda_{\Gamma_o} \to \FF\subset \PP(V) \times \PP(V^*) $$
which complete the proof of Proposition \ref{pro:constructphi}. \qed

\begin{figure}[ht!]
\centering
\includegraphics[scale=0.7]{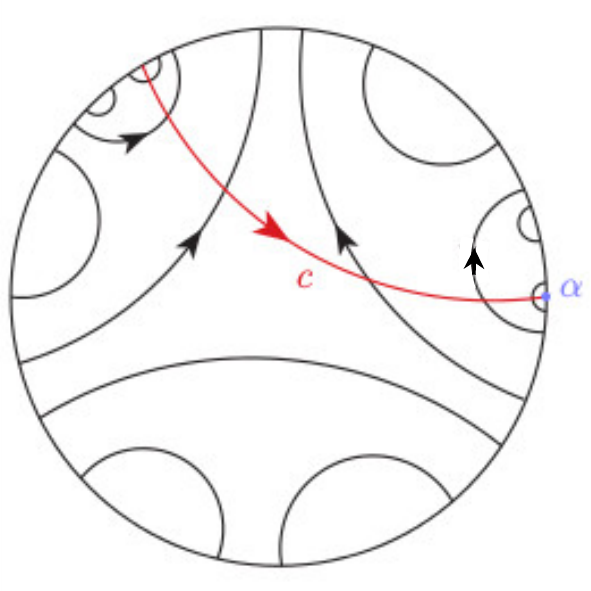}
\scriptsize
\put(-79,62){$\bar{\ell}_0$}
\put(-46,68){$\bar{\ell}_1$}
\put(-93,81){$\bar{\ell}_{-1}$}
\caption{A sequence $(\bar{\ell}_m)$ of leaves of $\mathcal{L}$ with reversed orientation.}
\label{fig:sequence_iL}
\end{figure}

\subsection{The Anosov property of new representations}

In this subsection, we construct the maps:
$$
\nu_+:\Omega(\phi^t)\subset T^1(\Gamma_o \backslash\mathbb{H}^2) \to Q(V) \quad \textrm{and} \quad \nu_-:\Omega(\phi^t)\subset T^1(\Gamma_o \backslash\mathbb{H}^2) \to Q(V^*)
$$
The definition is as follows: let $(x,v) \in \Omega(\phi^t)$ and let $c$ be the $\Gamma_o$-nonwandering oriented geodesic such that $c(0)=x$ and $c'(0)=v$. We denote by $c_-$ (resp. $c_+$) the tail (resp. head) of $c$.

If $x$ lies on a leaf $\ell$ of $\mathcal L$, which is oriented so that it is crossed by $c$ from the right to the left, then $\varphi_\lambda(c_+)$ (resp. $\varphi^{*}_\lambda(c_-)$) lies in the convex interior of the label $[\Theta]$ of $\ell$ (resp. in $\interior{[\Theta^*]}$). Define
\begin{itemize}
\item $\nu_+(x,v)$ as the Hilbert norm on $T_{\varphi_\lambda(c_+)}\PP(V)$ associated to $\interior{[\Theta]}$ in $\PP(V)$, and
\item $\nu_-(x,v)$ as the Hilbert norm on $T_{\varphi^{*}_\lambda(c_-)}\PP(V^*)$ associated to $\interior{[\Theta^*]}$ in $\PP(V^*)$.
\end{itemize}

Now if $x=c(0)$ does not lie on a leaf of $\mathcal{L}$, then let $c(-t_-)$ (resp. $c(t_+)$) be the first intersection point between $c$ and $\mathcal L$ in the past (resp. future). Observe that there exist uniform lower and upper bounds $\varepsilon_-$
and $\varepsilon_+$ of the time period, for which a nonwandering geodesic crosses a connected component of $\mathbb{H}^2 \setminus \mathcal{L}$, \ie
$\varepsilon_- \leq t_+ + t_- \leq \varepsilon_+$.
Define then $\nu_\pm(x,v)$ as the barycentric combination:
 $$\frac{t_-}{t_++t_-}\nu_\pm(c(t_+)) + \frac{t_+}{t_++t_-}\nu_\pm(c(-t_-))
 $$

Recall that $C>1$ is the uniform lower bound on the expansion of the Hilbert metrics when two leaves of $\mathcal{L}$ are crossed (see Definition \ref{def:C}). Let $N$ be the smallest integer such that $C^N >2$. It follows that the norm $\nu_+(c(t), c'(t))$ is at least doubled and $\nu_-(c(t), c'(t))$  divided by $2$ when $c$ crosses at least $2N$ leaves of $\mathcal L$. Moreover, this surely happens when one travels along $c$ for a time period $T = 2N\varepsilon_+$, and therefore the item $(ii)$ of Definition \ref{def:anosov} is satisfied (see also Remark \ref{rk:doubling}).

The proof of our main Theorem \ref{thm:main} is now complete.

\section{Extension of new representations to $\MG$}\label{conclusion}

In this section, we will give the proof of Theorem \ref{thm:main2}. In Sections \ref{sec:new} and \ref{sec:newisanosov}, we built a representation $\rho^\lambda_\Theta : \Gamma_o \to \HH$ for every marked box $[\Theta]$ and every $\lambda = (\varepsilon,\delta) \in \mathbb{R}^2$, and prove that if $\lambda \in \IR$, then $\rho^\lambda_\Theta$ is Anosov. In other words, we exhibit a subspace of $\Rep(\Gamma_o, \HH)$ which is made of Anosov representations and the boundary of which contains the restrictions to $\Gamma_o$ of the Schwartz representations. We now ask the following natural question:
\begin{center}
\emph{When does the representation $\rho^\lambda_\Theta: \Gamma_o \to \HH$ extend to a representation $\bar{\rho}^\lambda_\Theta : \Gamma \to \GG$?}
\end{center}

A main ingredient required for this extension is to find the image of the involution $I$: This image should be a polarity (see Remark \ref{rk:polarity}), and since we know the images of $R_*$ and $IR_*I$ under $\rho^\lambda_\Theta$, the problem of finding the image of $I$ reduces to:
\begin{center}
\emph{Find a polarity $\mathcal P$ such that $\mathcal{B}_{\Theta}^\lambda = \mathcal{P} \mathcal{A}_{\Theta}^\lambda \mathcal P$.}
\end{center}

As usual, equip $V$ with a $\Theta$-basis of $V$ and $V^*$ with its dual basis. Recall the proof of Lemma \ref{lemaprinc2}:  The projective transformation $\mathcal{A}_{\Theta}^\lambda$ does not depend on $\lambda$ and it corresponds to the matrix $A_{\Theta}$
in the proof of Lemma \ref{lemaprinc}. The projective transformation $\mathcal{B}_{\Theta}^\lambda$ is exactly $\NT^{-1}\mathcal{B}_{\Theta}^0\NT$
and $\mathcal{B}_{\Theta}^0 = \mathcal{D}_\Theta^0 \mathcal{A}_{\Theta}^0 \mathcal{D}_\Theta^0$, where $\mathcal{D}_\Theta^0$ corresponds to the matrix $D_{\Theta}$ in the proof of Lemma \ref{lemaprinc}. Since $\mathcal{B}_{\Theta}^0$ corresponds to the matrix:
$$B_\Theta^0 := D_\Theta^{-1} \; ^t(A_\Theta)^{-1} D_\Theta$$
the transformation $\mathcal{B}_{\Theta}^\lambda$ is represented by the matrix:
$$B_\Theta^\lambda := \NT^{-1} D_\Theta^{-1} \;^t(A_\Theta)^{-1} D_\Theta \NT$$
Now, the problem is to find an invertible symmetric matrix $S$ such that:
\begin{equation*}\label{eq:SS}
  S^{-1} \;^t(A_\Theta)^{-1} S = B_\Theta^\lambda
\end{equation*}
When $\Theta$ is special, the solution is easy: In this case, since $D_\Theta$ is the identity matrix, we simply let $S = \NT$.

From now on, assume that $\Theta$ is not special. In the appendix, we show through a computation that the existence of a non-zero symmetric matrix $S$
satisfying the equation
$${}^t(A_\Theta)^{-1} S = S B_\Theta^\lambda$$
is equivalent to:
\begin{equation}\label{eq:00}
  \det({\rm Id} - A_\Theta B_\Theta^\lambda)=0
\end{equation}
and by another computation, Equation \eqref{eq:00} holds if and only if:
\begin{equation*}
 0 = h (\varepsilon,\delta) := (\zeta_t^2 + \zeta_b^2 - 2 \zeta_t^2 \zeta_b^2) \, c_{\varepsilon} s_{\delta} (2 c_{\varepsilon} c_{\delta}  - s_{\varepsilon} ) -   \zeta_t \zeta_b (\zeta_t^2-\zeta_b^2) \, s_{\varepsilon}  ( c_{\varepsilon} c_{\delta}   - s_{\varepsilon}-1)
\end{equation*}
where $c_x = \cosh(x)$ and $s_y = \sinh(y)$.

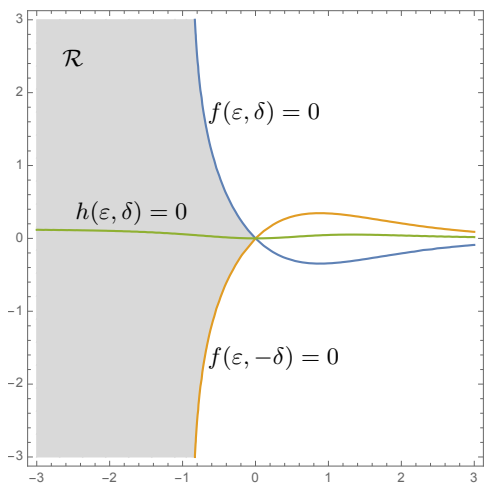
\begin{figure}[ht!]
\centering
\includegraphics[scale=0.5]{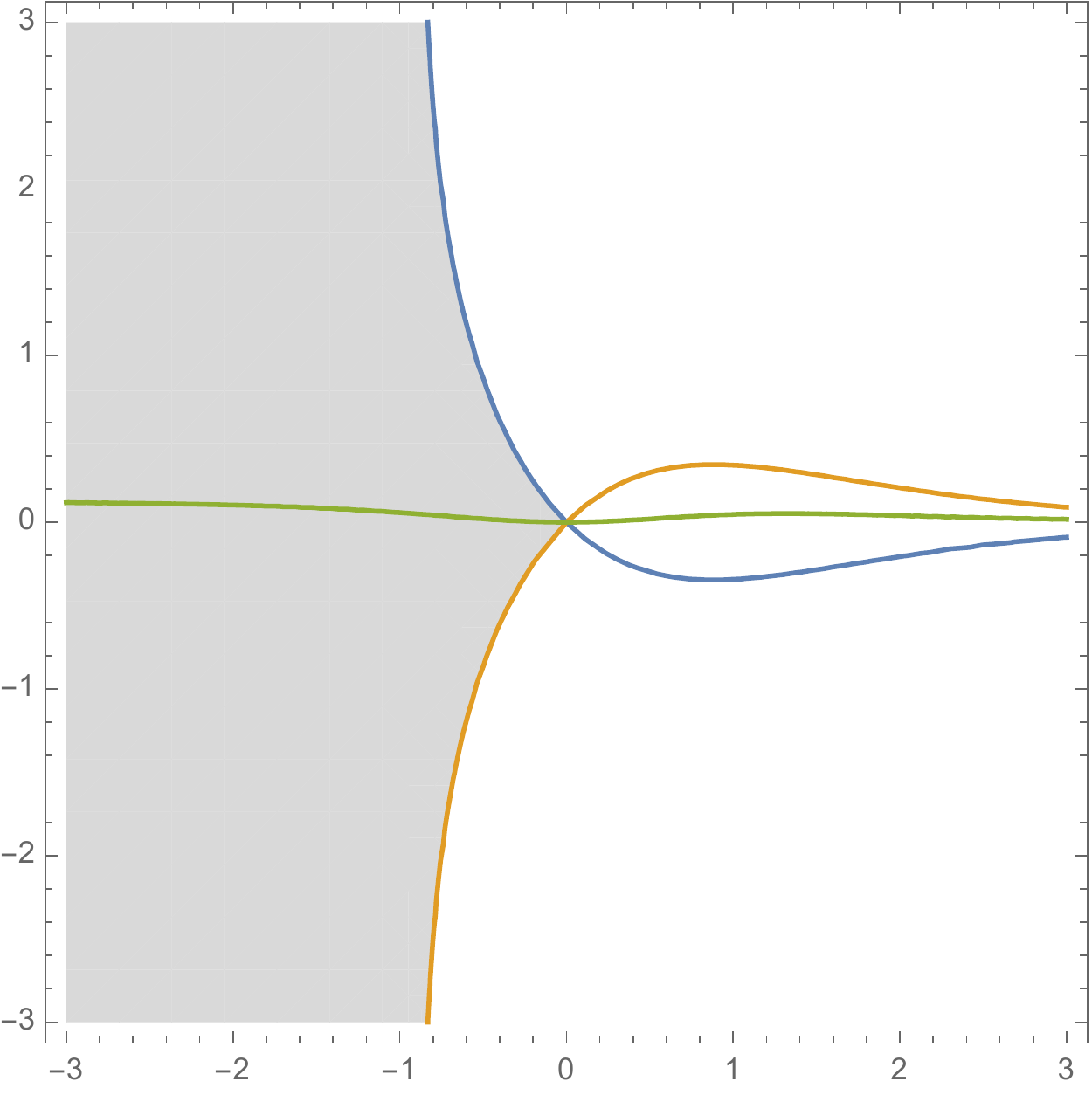}
\footnotesize
\put (-160, 160){$\mathcal{R}$}
\put (-155, 102){$h (\varepsilon,\delta) = 0$}
\put (-105, 140){$f(\varepsilon,\delta)= 0$}
\put (-105, 47){$f(\varepsilon,-\delta)= 0$}
\caption{The equation $h (\varepsilon,\delta) = 0$ is drawn in green.}
\label{fig:curve_h}
\end{figure}

Let $\mathcal{C} = \{(\varepsilon,\delta) \in \mathbb{R}^2 \mid h (\varepsilon,\delta) = 0 \}$ (see Figure \ref{fig:curve_h}). Since the invertibility of the matrix $S$ is an open condition, there exists an open neighborhood $\mathcal{U}$ of $(0,0)$ in $\mathbb{R}^2$ such that for every $ \lambda  \in \mathcal{C} \cap \mathcal{U}$, the representations $\rho^\lambda_\Theta$ extends to a representation $\bar{\rho}^\lambda_\Theta : \Gamma \to \GG$. Moreover, the following computation
$$
\frac{\partial f}{\partial \delta} (0,0) = -1 \quad \textrm{and} \quad \frac{\partial h}{\partial \delta} (0,0) = 2(\zeta_t^2 + \zeta_b^2 - 2 \zeta_t^2 \zeta_b^2) \neq 0
$$
and the implicit function theorem tell us that there exist an neighborhood $V_\varepsilon \times V_\delta$ of $(0,0)$ and two functions $\delta_f : V_\varepsilon \rightarrow \mathbb{R}$ and $\delta_h : V_\varepsilon \rightarrow \mathbb{R}$ such that:
 \begin{align*}
 \{   (\varepsilon, \delta_f(\varepsilon)) \mid  \varepsilon \in V_\varepsilon \} &= \{ (\varepsilon,\delta) \in  V_\varepsilon \times V_\delta \mid f(\varepsilon,\delta)=0 \}  \\
\{   (\varepsilon, \delta_h(\varepsilon)) \mid  \varepsilon \in V_\varepsilon \} &= \{ (\varepsilon,\delta) \in  V_\varepsilon \times V_\delta \mid h(\varepsilon,\delta)=0 \}
\end{align*}
Also, another simple computation
$$
\frac{d \delta_f}{d \varepsilon} (0) = -1 \quad \textrm{and} \quad \frac{d \delta_h}{d \varepsilon} (0) = 0
$$
shows that there exists an interval $\mathcal{V} := \,]\varepsilon_0, 0] \subset V_\varepsilon $ such that:
$$
(\varepsilon, \delta_h( \varepsilon)) \in \mathcal{R} \,\textrm{ for all }\, \varepsilon \in \mathcal{V} \quad \textrm{and} \quad
(\varepsilon, \delta_h( \varepsilon)) \in \IR \,\textrm{ for all }\, \varepsilon \in \interior{\mathcal{V}}
$$
 Therefore, if we let $\lambda = (\varepsilon, \delta_h(\varepsilon))$ for every $\varepsilon \in \mathcal{V}$, then the representation $\rho_{\Theta}^{\lambda} $ extends naturally to a representation $\bar{\rho}_{\Theta}^{\lambda} : \Gamma \to \GG$ when $\varepsilon \in \mathcal{V}$, it is Anosov when $\varepsilon \in \interior{\mathcal{V}}$, and it is the restriction of the Schwartz representation when $\varepsilon =0$.
It finishes the proof of the main Theorem \ref{thm:main2}.

\section{New representations in the representation variety}\label{section:InVariety}

In this section, we use the same notation as in Section \ref{conclusion} and show that the $\HH$-orbit of new representations in $\Rep(\Gamma_o, \HH)$ has a non-empty interior.

We denote by $\pi : \GL(3,\mathbb{R}) \rightarrow \mathrm{SL}(3,\mathbb{R})$ the composition of the projection $\GL(3,\mathbb{R}) \rightarrow \PGL(3,\mathbb{R})$ with the natural isomorphism $\PGL(3,\mathbb{R}) \simeq \mathrm{SL}(3,\mathbb{R})$, and identify $\HH$ with $\mathrm{SL}(3,\mathbb{R})$.

\begin{lemma}\label{lemma:C-H}
Let $A \in \mathrm{SL}(3,\mathbb{R})$. If $A \neq \mathrm{Id}$, then:
$$ A^3 = \mathrm{Id} \quad \textrm{if and only if} \quad \mathrm{tr}(A) = \mathrm{tr}(A^{-1}) =0 $$
\end{lemma}
\begin{proof}
It follows from Cayley-Hamilton theorem (see e.g. Acosta \cite[Lemma 4.2]{Acosta}).
\end{proof}

We denote by $\mathrm{M}(3,\mathbb{R})$ the set of $3 \times 3$ real matrices. Define a map
$$ \Psi : \mathrm{M}(3,\mathbb{R}) \times \mathrm{M}(3,\mathbb{R}) \rightarrow \mathbb{R}^6 $$
by assigning to any pair of matrices $( A := (A_{ij}), B := (B_{ij}))$ the $6$-tuple of polynomials $(\Psi_i(A,B) )_{i = 1, \dotsc, 6}$, where:
$$
{\small
\begin{tabular}{lll}
$\Psi_{1} = \det(A)-1$, & $\Psi_{2} = \mathrm{tr}(A)$, & $\Psi_{3} = A_{11}A_{22}+A_{22}A_{33}+A_{33}A_{11} - A_{12}A_{21} -A_{23}A_{32} - A_{31}A_{11}$ \\
$\Psi_{4} = \det(B)-1$, & $\Psi_{5} = \mathrm{tr}(B)$, & $\Psi_{6} = B_{11}B_{22}+B_{22}B_{33}+B_{33}B_{11} - B_{12}B_{21} -B_{23}B_{32} - B_{31}B_{11}$ \\
\end{tabular}}
$$

Observe that $\Psi_{3} = \mathrm{tr}(A^{-1})$ and $\Psi_{6} = \mathrm{tr}(B^{-1})$ for every $A, B \in \mathrm{SL}(3,\mathbb{R})$. By Lemma \ref{lemma:C-H}, the real algebraic variety $\Psi^{-1}(0)$ is isomorphic to a union of components of $\Rep(\Gamma_o, \HH)$ (cf. Lawton \cite{Lawton}).

\begin{lemma}\label{lemma:Tangentspace}
For every $\lambda \in \mathcal{R}$ and every convex marked box $[\Theta]$, the representation $\rho^\lambda_{\Theta}$ is a smooth point of $\Rep(\Gamma_o, \HH)$.
\end{lemma}

\begin{proof}
We claim that $\Psi^{-1}(0)$ is smooth at $(A,B) = (\pi(A_{\Theta}), \pi(B_{\Theta}^{\lambda}))$  for every $\lambda \in \mathcal{R}$ and every convex marked box $[\Theta]$. Indeed, consider a map
 $$
\tilde{\Psi} : \mathrm{M}(3,\mathbb{R}) \times \mathrm{M}(3,\mathbb{R}) \rightarrow \mathbb{R}^{6} \times \mathbb{R}^{6} \times \mathbb{R}^{6} $$ given by
$
\tilde{\Psi}(A,B) = \left( \Psi(A,B), (A_{ij})_{i \neq j}, (B_{ij})_{i \neq j} \right)
$.
A computation shows that if $\lambda = (\varepsilon,\delta) \in \mathcal{R}$ then $\varepsilon \leq 0$, and that:
\begin{multline*}
\left| \det\left( D\tilde{\Psi}|_{(A,B) = (\pi(A_{\Theta}), \pi(B_{\Theta}^{\lambda})) } \right) \right| \\
= \tfrac{9  (1+\zeta_t \zeta_b) (1-\zeta_t \zeta_b)^2 \left( 2 \cosh(2\varepsilon) (1+\zeta_t\zeta_b ) - \sinh(2\varepsilon) \left(e^{-\delta}(2+\zeta_t\zeta_b-\zeta_t^2)  + e^{\delta} (2+\zeta_t\zeta_b-\zeta_b^2) \right)  \right)}{2 (1-\zeta_t^2)^2 (1-\zeta_b^2)^2 }
\end{multline*}
which is non-zero because $-1 < \zeta_t, \zeta_b < 1$ and $\varepsilon \leq 0$. As a consequence, the points $(\pi(A_{\Theta}), \pi(B_{\Theta}^{\lambda}))$ on $\Psi^{-1}(0)$ are non-singular, which completes the proof.
\end{proof}

\begin{theorem}
Let $[\Theta_0]$ be a non-special convex marked box. Then there exists an open neighbourhood $\mathcal{U}$ of the Schwartz representation $\rho_{\Theta_0}$ in $\Rep(\Gamma_o, \HH)$ such that every representation in $\mathcal{U}$ is conjugate to a representation $\rho^\lambda_{\Theta}$ for some convex marked box $\Theta$ and some $\lambda \in \mathbb{R}^2$.
\end{theorem}

\begin{proof}
It is enough to show that a map $$\Phi : {]\!-\!1, 1[}^2 \times \mathbb{R}^2 \times \mathrm{SL}(3,\mathbb{R}) \rightarrow \Psi^{-1}(0) \subset \mathrm{M}(3,\mathbb{R}) \times \mathrm{M}(3,\mathbb{R}) $$
given by
$\left( (\zeta_t, \zeta_b),\; \lambda,\; g \right) \mapsto (g \,\pi(A_{\Theta})\, g^{-1},\; g \,\pi(B_{\Theta}^{\lambda})\, g^{-1})$
is a local diffeomorphism at any point $p$ of
 $\mathcal{V} := {]\!-\!1, 1[}^2 \backslash \{(0,0)\} \times (0,0) \times \mathrm{Id}$.
Consider another map $\tilde{\Phi} : {]\!-\!1, 1[}^2 \times \mathbb{R}^2 \times \GL(3,\mathbb{R}) \rightarrow \mathbb{R}^6 \times \mathbb{R}^6 \times \mathbb{R}$ given by:
$$ \left( \left(\zeta_t, \zeta_b\right),\; \lambda,\; g \right) \mapsto \left( \left( \left(g \,\pi(A_{\Theta})\, g^{-1} \right)_{ij} \right)_{i \neq j}, \; \left( \left(g \,\pi(B^{\lambda}_{\Theta})\, g^{-1} \right)_{ij} \right)_{i \neq j}, \; \det(g) \right)$$
A computation shows that:
$$
\left| \det\left( D\tilde{\Phi}|_{( (\zeta_t, \zeta_b),\, \lambda,\, g) = ((\zeta_t, \zeta_b),\, (0,0),\, \mathrm{Id}) } \right) \right| = \tfrac{288 (1-\zeta_t^2 \zeta_b^2)^2 (2 - \zeta_t^2 - \zeta_b^2) \left( \zeta_t^2 (1 - \zeta_b^2) + \zeta_b^2 (1- \zeta_t^2)\right)}{ (1-\zeta_t^2)^5 (1-\zeta_b^2)^5 } \\
$$
which is not zero because $-1 < \zeta_t, \zeta_b < 1$ and $(\zeta_t, \zeta_b) \neq (0,0)$. As a consequence, the map $\tilde{\Phi}$ is a local diffeomorphism at any point $p$ of $\mathcal{V}$. Since the map $\tilde{\Psi}$ in the proof of Lemma \ref{lemma:Tangentspace} is a local diffeomorphism at $(\pi(A_{\Theta}), \pi(B_{\Theta}^{\lambda}))$ (and therefore at $(g\pi(A_{\Theta})g^{-1}, g\pi(B_{\Theta}^{\lambda})g^{-1}$) for every $\lambda \in \mathcal{R}$ and every convex marked box $[\Theta]$, the map $\Phi$ is also a local diffeomorphism at $p$.
\end{proof}

\section*{Appendix}

A matrix $A$ in $\GL(3,\mathbb{R})$ is a \emph{rotation of angle} $\theta$ if there exists a matrix $Q$ in $\GL(3,\mathbb{R})$ such that $Q^{-1} A Q = \mu R_{\theta}$, where $\mu \neq 0$ and
\begin{equation*}\label{rotation}
R_{\theta} = \left(\begin{array}{ccc}
                      1 & 0 & 0 \\
                      0 & \cos(\theta) & -\sin(\theta) \\
                      0 & \sin(\theta) & \cos(\theta) \\
                    \end{array}
\right).
\end{equation*}

\begin{lemma}\label{appendix}
Let $A$ be a rotation of angle $\theta$. Assume that $0 < \theta < \pi$ and
$B = G^{-1} \trans{A}^{-1} G$ for some $G \in \GL(3,\mathbb{R})$. Then $ \det ({\rm Id} - A B ) = 0$ if and only if there exists a symmetric matrix $S \neq 0$ such that $S B = \trans{A}^{-1} S$.
\end{lemma}

\begin{proof}
By the assumption, there exists a matrix $Q$ in $\GL(3,\mathbb{R})$ such that $Q^{-1} A Q = \mu R_{\theta}$, and so $\trans{Q} \trans{A}^{-1} \trans{Q}^{-1} = \mu^{-1} R_{\theta}$. This implies that:
\begin{align*}
S B = \trans{A}^{-1} S & \Leftrightarrow S G^{-1} \trans{A}^{-1} G = \trans{A}^{-1} S \\
 & \Leftrightarrow (\trans{Q} S Q)  (\trans{Q} G Q)^{-1}  R_{\theta} (\trans{Q} G Q) = R_\theta  (\trans{Q} S Q)
\end{align*}
As a consequence, there exists a non-zero symmetric matrix $S$ satisfying
$SB  = \trans{A}^{-1} S $ if and only if there exists a symmetric matrix $P \neq 0$ such that:
 $$ P(\trans{Q} G Q )^{-1} R_{\theta}  =   R_\theta P(\trans{Q} G Q )^{-1} $$ It follows that $R_\theta$ commutes with $P (\trans{Q} G Q)^{-1}$, and therefore
\begin{equation*}
P = \left(\begin{array}{ccc}
                      \alpha & 0 & 0 \\
                      0 & \beta  & -\gamma \\
                      0 & \gamma  & \beta \\
                    \end{array}
\right) (\trans{Q} G Q) \quad \textrm{for some } \alpha, \beta, \gamma  \in \mathbb{R}.
\end{equation*}
If $U = (U_{ij})_{i,j = 1,2,3}$ denotes $\trans{Q} G Q$, then
we can write the equation $P - \trans{P} = 0$ as follows:
\begin{equation}\label{matrix}
\left(\begin{array}{ccc}
                      -U_{12} & U_{21} & -U_{31}          \\
                      -U_{13} & U_{31} & U_{21}          \\
                      0        & U_{32}-U_{23}   & U_{22}+U_{33} \\
                    \end{array}
\right)
\left(\begin{array}{c}
                      \alpha \\
                      \beta  \\
                      \gamma \\
                    \end{array}
\right)
= \left(\begin{array}{c}
                      0 \\
                      0  \\
                      0 \\
                    \end{array}
\right)
\end{equation}
Let $M$ be the left $3 \times 3$ matrix of Equation (\ref{matrix}). Then by a simple computation, we have:
\begin{align*}
2 \sin (\theta) (1 - \cos (\theta)) \cdot \det(M) & =  {\det}(U) \cdot \det({\rm Id} - R_{\theta} U^{-1} R_{\theta} U) \\
& =  {\det}(U) \cdot \det({\rm Id} - A G^{-1} \trans{A}^{-1} G)
\end{align*}
In the last step, we use the fact that:
$$
A G^{-1} \trans{A}^{-1} G = Q R_{\theta} U^{-1} R_{\theta} U Q^{-1}
$$
Finally, $\det(M)=0$ if and only if $\det({\rm Id} - A B)=0$. The result follows.
\end{proof}

\begin{remark}
One implication in Lemma \ref{appendix} is easier to prove without  computation. If $B = S^{-1} \trans{A}^{-1} S$ with $S$ an invertible symmetric matrix, then:
\begin{align*}
{\rm Id}-AB & = {\rm Id} -AS^{-1}\trans{A}^{-1} S \\
                 & = AS^{-1}  (SA^{-1} - \trans{A}^{-1} S)   \\
                 & =AS^{-1}  (SA^{-1} - \trans{(SA^{-1})})\quad \textrm{($S$ is symmetric)}
\end{align*}
Notice that $SA^{-1} - \trans{(SA^{-1})}$ is an anti-symmetric $3 \times 3$ matrix, which implies that:
$$\det({\rm Id} -AB)=0$$
\end{remark}


\end{document}